\documentclass[12pt,a4paper]{article}
\usepackage[utf8]{inputenc}

\usepackage{amsfonts,amsmath, amssymb}
\usepackage{amsthm}
\usepackage{graphicx}
\usepackage{longtable}
\usepackage{float}
\usepackage{array}
\usepackage{enumerate}
\usepackage{threeparttable}
\usepackage{mathrsfs}
\usepackage[text={15cm,24cm},top=2cm,margin=2cm]{geometry} %layout
\usepackage{listings}
\usepackage{url}
\usepackage{subfigure}
\usepackage{multirow}
\usepackage{textcomp,booktabs}
\usepackage[usenames,dvipsnames]{color}
\usepackage{colortbl}
\usepackage{lscape}
\usepackage{bm}
\usepackage[numbers,square]{natbib}%,authoryear%bibtex
\usepackage[colorlinks=true,urlcolor=black,citecolor=blue,linkcolor=black,bookmarks=true]{hyperref} %index
\usepackage{verbatim}
\usepackage{ulem}
\usepackage{tikz}
\usetikzlibrary{positioning,arrows}
\tikzset{
	block/.style={
		draw, 
		rectangle, 
		minimum height=1cm, 
		minimum width=2cm, align=center
	}, 
	line/.style={->,>=latex'}
}

\definecolor{mygray}{gray}{.9}
\usepackage[labelfont={bf,it},textfont={bf,it}]{caption}
\DeclareGraphicsExtensions{.jpg,.pdf,.gif,.png}

\newtheorem{proposition}{Proposition}
\newtheorem{theorem}{Theorem}

\newtheorem{lemma}{Lemma}

\newcommand{\PP}{\mathbb{P}}

\newcommand{\R}{\mathbb{R}}

\DeclareMathOperator\tr{tr}

\DeclareMathOperator\erf{erf}
\DeclareMathOperator\sgn{sgn}

\numberwithin{equation}{section}
\numberwithin{definition}{section}
\numberwithin{remark}{section}
\numberwithin{theorem}{section}
\numberwithin{proposition}{section}
\numberwithin{lemma}{section}
\numberwithin{remark}{section}
\numberwithin{example}{section}
\numberwithin{figure}{section}
\numberwithin{conjecture}{section}
\numberwithin{table}{section}

\begin{document}
\pagestyle{plain}
\title{\bf Upscaling between an Agent-Based Model (Smoothed Particle Approach) and a Continuum-Based Model for Skin Contractions}
\author{Q. Peng, F.J. Vermolen}
\maketitle

\begin{abstract}
Skin contraction is an important biophysical process that takes place during and after recovery of deep tissue injury. This process is mainly caused by fibroblasts (skin cells) and myofibroblasts (differentiated fibroblasts) that exert pulling forces on the surrounding extracellular matrix (ECM). Modelling is done in multiple scales: agent--based modelling on the microscale and continuum--based modelling on the macroscale. In this manuscript we present some results from our study of the connection between these scales. For the one--dimensional case, we managed to rigorously establish the link between the two modelling approaches for both closed--form solutions and finite--element approximations. For the multi--dimensional case, we computationally evidence the connection between the agent--based and continuum--based modelling approaches. 
 
 {\bf Key Words:} traction forces, mechanics, smoothed particle approach, agent-based modelling, finite element methods, continuum-based modelling
\end{abstract}

\section{Introduction}
\noindent 
Wound healing is a spontaneous process for the skin to cure itself after an injury. It is a complicated combination of various cellular events that contribute to resurfacing, reconstituting and restoring of the tensile strength of the injured skin. For superficial wounds that only stay in epidermis, the wound can be healed without any problem. However, for severe injuries, in particular dermal wound, they may result in various pathological problems, such as contractures. 

Contractures concur with disabilities and disfunctioning of patients, which cause a significantly severe impact on patients' daily life. Contractures are recognized as excessive and problematic contractions, which occur due to the pulling forces exerted by the (myo)fibroblasts on their direct environment, i.e. the extra cellular matrix (ECM). Contractions mainly start occurring in the proliferation phase of wound healing: the fibroblasts are migrating towards the wound and differentiating into myofibroblasts due to the high concentration of transforming growth factor-beta (TGF-beta). Usually, $5-10\%$ reduction of wound area has been observed in clinical trials. A more detailed biological description can be found in \cite{enoch2008basic, cumming2009mathematical,haertel2014transcriptional,Martin1997}.

In our previous work \cite{Peng2020}, a formalism to describe the mechanism of the displacement of the ECM has been used, which is firstly developed by \citet{boon2016multi} and improved further by \citet{koppenol2017biomedical}. Regarding the elasticity equation with point forces, we realized that the solution to the partial differential equation is singular for dimensionality exceeding one. Hence, we developed various alternatives to improve the accuracy of the solution in \cite{peng2019numerical, peng2019pointforces}. 

We have been working with agent-based models so far, which model the cells as individuals and define the formalism of pulling forces by superposition theory. However, once the wound scale is larger, the agent-based model is increasingly expensive from a computational perspective, and hence, the cell density model is preferred, which considers many cells as one collection in a unit. In this manuscript, we investigate and discover the connections between these two models, in the perspective of modelling the mechanism of pulling forces exerted by the (myo)fibroblasts. As the consistency between the smoothed particle approach (SPH approach) and the immersed boundary approach has been proven both analytically and numerically \cite{peng2019numerical, peng2019pointforces}, we select the SPH approach here due to its continuity and smoothness, to compare with the cell density model using finite-element methods.

The manuscript is structured as follows. We start introducing both models in one dimension in Section \ref{MathsModel_1D}, then in Section \ref{MathsModel_2D} we extend the models to two dimensions. Section \ref{Results} displays the numerical results in one and two dimensions. Finally, some conclusions are shown in Section \ref{Conclusions}.     

\section{Mathematical Models in One Dimension}\label{MathsModel_1D}
\noindent
Considering one-dimensional force equilibrium, the equations are given by
\begin{eqnarray*}
	-\frac{d\sigma}{dx}&=f, &\quad\mbox{Equation of Equlibirum,}\\
	\epsilon&=\frac{du}{dx}, &\quad\mbox{Strain-Displacement Relation,}\\
	\sigma&=E\epsilon, &\quad\mbox{Constitutive Equation.}
\end{eqnarray*}
By substituting $E=1$, the equations above can be combined to Laplacian equation in one dimension:
\begin{equation}
-\frac{d^2u}{dx^2}=f.
\end{equation}

\subsection{Smoothed Particle Approach}
\noindent
In  \cite{peng2019pointforces}, a smoothed particle approach (SPH approach) is developed as an alternative of the Dirac Delta distribution describing the point forces exerted by the biological cells, in the application of wound healing:
\begin{equation}
\label{Eq_SPH_1d}
(BVP_{SPH})\left\lbrace
\begin{aligned}
-\frac{d^2u}{dx^2}&=P_{SPH}\sum_{i=1}^{N_s}\delta'_{\varepsilon}(x-s_i), \mbox{$x\in(0,L)$,}\\
u(0)&=u(L)=0,
\end{aligned}
\right.
\end{equation}
where $P_{SPH}$ is the magnitude of the forces, $\delta_{\varepsilon}(x)$ is the Gaussian distribution with variance $\varepsilon$ and $s_i$ is the centre position of biological cell $i$. One can solve the partial differential equations (PDEs) with finite-element methods. The corresponding weak form is given by
\[ (WF\textsubscript{SPH})\left\{
\begin{aligned}
&\text{Find $u\in H_0^1((0, L))$, such that}\\
&\int_0^L u'\phi'dx=\int_0^L\sum_{i=1}^{N_s}P_{SPH}\delta'_{\varepsilon}(x-s_i)\phi dx,\\
&\text{for all $\phi\in H_0^1((0, L))$.}
\end{aligned}
\right.\]

%The existence and uniqueness of $u$ in $H^1_0((0,L))$ is evident since the solution to the original boundary value problem is continuous and piecewise linear. We will see this in the next section.

Without this knowledge, the existence and uniqueness of the
$H^1_0$-solution follows as well from the application of the Lax--Milgram theorem \cite{braess2007finite}, where it is immediately obvious that the bilinear form in the left--hand side is symmetric and positive definite.
%Boundedness of the right-hand side follows from
%$$
%| \int_0^L \delta(x-a) \phi(x) dx | = | \phi(a) | = | \int_0^a \phi'(x) dx | \le \sqrt{a} \left[ \int_0^L [ \phi'(x) ]^2dx \right]^{1/2}  \le \sqrt{L} || \phi(x) ||_{H^1((0,L))}.
%$$
%Note that we worked this out for one delta--source only. Superposition theory implies that boundedness also follows for multiple Delta sources.

\subsection{Cell Density Approach}
\noindent
A cell density approach is often used in the large scale, so that the computational efficiency is much improved compared with the agent-based model. According to the model in \cite{koppenol2017biomedical}, the force in two dimensions can be determined by the divergence of $n_c\cdot\boldsymbol{I}$, where $n_c$ is the local density of the biological cells and $\boldsymbol{I}$ is the identity tensor. In one dimension, the cell density approach is expressed as:
\begin{equation}
\label{Eq_density_1d}
(BVP_{den})\left\lbrace
\begin{aligned}
-\frac{d^2u}{dx^2}&=P_{den}\frac{dn_c}{dx}, \mbox{$x\in(0,L)$,}\\
u(0)&=u(L)=0,
\end{aligned}
\right.
\end{equation}
where $P_{den}$ is the magnitude of the forces. The corresponding weak form is given by
\[ (WF\textsubscript{den})\left\{
\begin{aligned}
&\text{Find $u\in H_0^1((0, L))$, such that}\\
&\int_0^L u'\phi'dx=\int_0^L P_{den}n'_s\phi dx,\\
&\text{for all $\phi\in H_0^1((0, L))$.}
\end{aligned}
\right.\]

\subsection{Consistency between Two Models}
\subsubsection{Analytical Solutions with Specific Locations of Biological Cells}
\noindent
To express the analytical solution, it is necessary to determine the locations of the biological cells, such that the cell density can be written as an analytical function of the positions. We assume, there are $N_s$ cells distributed uniformly in the subdomain $(a, b)$ of the computational domain $(0, L)$. Hence, the distance between the center position of any two adjacent biological cells is constant, which we denote $\Delta s = (b-a)/N_s$ and the first and the $N_s$-th cell are located at $x=a+\Delta s/2$ and $x=b-\Delta s/2$, respectively. With homogeneous Dirichlet boundary conditions, and suppose $P_{SPH} = P\Delta s$ and variance $\varepsilon = \Delta s$, the boundary value problem of the SPH approach is expressed as
\begin{equation}
\label{Eq_SPH_1d_exact}
(BVP^1_{SPH})\left\lbrace
\begin{aligned}
-\frac{d^2u_1}{dx^2}&=P\Delta s\sum_{i=1}^{N_s}\delta'_{\Delta s}(x-s_i), \mbox{$x\in(0,L)$,}\\
u_1(0)&=u_1(L)=0,
\end{aligned}
\right.
\end{equation}
where $P$ is a positive constant and $s_i$ is the centre position of the biological cells. Utilizing the superposition principle, the analytical solution is given by
\begin{equation}
\label{Eq_SPH_1d_sol}
u_1(x) = P\Delta s\sum_{i=1}^{N_s}\frac{1}{2}\left\{\left(\frac{x}{L}-1\right)\erf\left(\frac{s_i}{\sqrt{2}\Delta s}\right)+\frac{x}{L}\erf\left(\frac{L-s_i}{\sqrt{2}\Delta s}\right)-\erf\left(\frac{x-s_i}{\sqrt{2}\Delta s}\right) \right\},
\end{equation}
where $\erf(x)$ is the error function defined as $\erf(x)=\frac{2}{\sqrt{\pi}}\int_{0}^{x}\exp(-t^2)dt$ \cite{weisstein2010erf}. Since the biological cells are uniformly located between $a$ and $b$  ($0<a<b<L$),  $\frac{dn_c}{dx}$ can be rephrased as
\begin{equation*}
\frac{dn_c}{dx} =
\left\{
\begin{aligned}
&\frac{1}{t}, &a-\frac{t}{2}< x< a+\frac{t}{2},\\
&-\frac{1}{t}, &b-\frac{t}{2}<x<b+\frac{t}{2},\\
&0, &\mbox{otherwise,}
\end{aligned}
\right.
\end{equation*}
where $t$ is a small positive constant. Taking $t$ to zero, the above expression converges to $\delta(x-a)-\delta(x-b)$. Hence, the boundary value problem of the cell density model can be written as
\begin{equation}
\label{Eq_density_1d_exact}
(BVP^1_{den})\left\lbrace
\begin{aligned}
-\frac{d^2u_2}{dx^2}&=P\frac{dn_c}{dx}\rightarrow P(\delta(x-a)-\delta(x-b)), \mbox{$x\in(0,L)$,}\\
u_2(0)&=u_2(L)=0,
\end{aligned}
\right.
\end{equation}
where $\delta(x)$ is the Dirac Delta distribution and $a$ and $b$ are the left and right endpoint of the subdomain (where biological cells are uniformly located) respectively.  The analytical solution is then expressed as
\begin{equation}
\label{Eq_density_1d_sol}
u_2(x) = P(G(x,a) - G(x,b)),
\end{equation}
where $G(x, x')$ is the Green's function \cite{haberman1983elementary}, defined by $$G(x, x') = (1-\frac{x'}{L})x-max(x-x', 0),$$ in the computational domain $(0,L)$.

Actually, the convergence between $u_1(x)$ and $u_2(x)$ can be proven as $\Delta s\rightarrow 0^+$ by a proposition. Firstly, we introduce Chebyshev's Inequality:
\begin{lemma}{\bf (Chebyshev's Inequality \citep{Olkin1958})}\label{lemma_chebyshev_chap-6}
	Denote $X$ as a random variable with finite mean $\mu$ and finite variance $\sigma^2$. Then for any positive $k\in\R$, the following inequality holds:
	$$\PP(|X-\mu|\geqslant k)\leqslant \frac{\sigma^2}{k^2},$$
	where $\PP(A)$ is the probability of event $A$. The above inequality can also be rephrased as $$\PP(|X-\mu|\leqslant k)\geqslant 1-\frac{\sigma^2}{k^2}.$$
\end{lemma}

\begin{proposition}\label{Prop_exact_consis_1D}
	Let $u_1(x)$ as described in Eq (\ref{Eq_SPH_1d_exact}) be the exact solution to $(BVP^1_{SPH})$ and $u_2(x)$ as described in Eq (\ref{Eq_density_1d_exact}) be the exact solution to $(BVP_{den}^1)$. As $\Delta s\rightarrow 0^+$, $u_1(x)$ converges to $u_2(x)$.
\end{proposition}
\begin{proof}
	For the standard Gaussian distribution in one dimension, the cumulative distribution function is given by $\displaystyle F(x)=\frac{1}{2}\left(1+\erf\left(\frac{x}{\sqrt{2}}\right)\right)$. Thus, we obtain
	\begin{equation}
	\label{Eq_erf_F}
	\erf\left(\frac{x}{\sqrt{2}}\right) = 2F(x) - 1.
	\end{equation}
	By Chebyshev's Inequality (see Lemma \ref{lemma_chebyshev_chap-6}), one can conclude that for any positive $k$,
	\begin{equation}
	\label{Eq_F_inequality}
	F(k)-F(-k)\geqslant 1-\frac{1}{k^2}.
	\end{equation}
	Note that $1-F(k) = F(-k)$ due to the symmetry of standard Gaussian distribution. Hence, Eq (\ref{Eq_F_inequality}) implies
	\begin{align*}
	&F(k)\geqslant 1-\frac{1}{k^2}+F(-k) = 1-\frac{1}{k^2}+1-F(k)\\
	\Leftrightarrow & 1-\frac{1}{k^2}\leqslant F(k)\leqslant 1,
	\end{align*}
	and analogously, $0\leqslant F(-k)\leqslant\frac{1}{k^2}$ is implied. Together with Eq (\ref{Eq_erf_F}), it gives
	\begin{equation*}
	\left\{
	\begin{aligned}
	&1-\frac{1}{k^2}\leqslant \erf\left(\frac{k}{\sqrt{2}}\right)\leqslant1,\\
	&-1\leqslant \erf\left(-\frac{k}{\sqrt{2}}\right)\leqslant -1+\frac{1}{k^2}.
	\end{aligned}
	\right.
	\end{equation*}
	Let $k=\frac{s_i}{\Delta s}>0$, for any $s_i\in(a,b)\subset(0,L), i=\{1,\dots, N_s\}$, then
	\begin{equation*}
	\left\{
	\begin{aligned}
	&1-\left(\frac{\Delta s}{s_i}\right)^2\leqslant \erf\left(\frac{s_i}{\sqrt{2}\Delta s}\right)\leqslant1,\\
	&-1\leqslant \erf\left(-\frac{s_i}{\sqrt{2}\Delta s}\right)\leqslant -1+\left(\frac{\Delta s}{s_i}\right)^2.
	\end{aligned}
	\right.
	\end{equation*}
	As it has been defined earlier that $\Delta s = (b-a)/N_s$, we obtain
	\begin{align*}
	&\sum_{i=1}^{N_s}\left(1-\left(\frac{\Delta s}{s_i}\right)^2\right)\Delta s\leqslant \sum_{i=1}^{N_s} \erf\left(\frac{s_i}{\sqrt{2}\Delta s}\right)\Delta s\leqslant \sum_{i=1}^{N_s}\Delta s\\
	\Rightarrow & (b-a) - (\Delta s)^3\sum_{i=1}^{N_s}\frac{1}{s_i^2}\leqslant \sum_{i=1}^{N_s} \erf\left(\frac{s_i}{\sqrt{2}\Delta s}\right)\Delta s\leqslant (b-a).
	\end{align*}
	Since $\displaystyle\lim_{\Delta s\rightarrow 0^+}(\Delta s)^3\sum_{i=1}^{N_s}\frac{1}{s_i^2}=0$ for any $s_i\in (a,b)\subset(0,L)$, the Squeeze Theorem \cite{apostol1958mathematical} implies that
	\begin{equation}
	\label{Eq_erf_lim_1}
	\displaystyle\lim_{\Delta s\rightarrow0^+}\sum_{i=1}^{N_s}\erf\left(\frac{s_i}{\sqrt{2}\Delta s}\right)\Delta s = b-a.
	\end{equation}
	Analogously, we obtain that for any $s_i\in (a,b)\subset(0,L)$,
	\begin{equation}
	\label{Eq_erf_lim_2}
	\displaystyle\lim_{\Delta s\rightarrow0^+}\sum_{i=1}^{N_s}\erf\left(-\frac{s_i}{\sqrt{2}\Delta s}\right)\Delta s = a-b.
	\end{equation}
	Thus, it can be concluded that for any series of real number $\{x_i\}\in\R^n$, when $x_{i+1}-x_i=\Delta s$ and $x_i$ is either all positive or all negative for any $i=\{1, \cdots, N_s\}$,
	\begin{equation}
	\label{Eq_erf_lim_3}
	\displaystyle\lim_{\Delta s\rightarrow0^+}\sum_{i=1}^{N_s}\erf\left(\frac{x_i}{\sqrt{2}\Delta s}\right)\Delta s = (b-a)\sgn(x_i),
	\end{equation}
	where $\sgn(x)$ is sign function defined by
	\begin{equation*}
	\sgn(x) = \left\{
	\begin{aligned}
	&1, \mbox{ if $x>0$,}\\
	&0, \mbox{ if $x=0$,}\\
	&-1, \mbox{ if $x<0$.}\\
	\end{aligned}
	\right.
	\end{equation*}
	
	We rewrite $u_1(x)$ as
	\begin{align*}\displaystyle
	& u_1(x) =  P\left[\frac{1}{2}\left(\frac{x}{L}-1\right)\sum_{i=1}^{N_s}\erf\left(\frac{s_i}{\sqrt{2}\Delta s}\right)\Delta s + \frac{1}{2}\frac{x}{L}\sum_{i=1}^{N_s}\erf\left(\frac{L-s_i}{\sqrt{2}\Delta s}\right)\Delta s - \frac{1}{2}\sum_{i=1}^{N_s}\erf\left(\frac{x-s_i}{\sqrt{2}\Delta s}\right)\Delta s \right].\\
	%  &\Rightarrow \lim_{\Delta s\rightarrow 0^+} u_1(x) = \lim_{\Delta s\rightarrow 0^+} P\left[\frac{1}{2}\left(\frac{x}{L}-1\right)\int_{a}^{b}\erf\left(\frac{s}{\sqrt{2}\Delta s}\right)ds + \frac{1}{2}\frac{x}{L}\int_{a}^{b}\erf\left(\frac{L-s}{\sqrt{2}\Delta s}\right)ds \right.\\
	%  & \left.- \frac{1}{2}\int_a^b\erf\left(\frac{x-s}{\sqrt{2}\Delta s}\right)ds\right]\\
	%  & = P\left[\frac{1}{2}\left(\frac{x}{L}-1\right)\int_{a}^{b}1ds + \frac{1}{2}\frac{x}{L}\int_{a}^{b}1ds - \frac{1}{2}\int_a^b\sgn(x-s)ds\right]\\
	%  & = P\left[\left(\frac{x}{L}-\frac{1}{2}\right)(b-a)+\frac{1}{2}[(x-b)\sgn(x-b)-(x-a)\sgn(x-a)]\right]\\
	%  & = P\left[\left(\frac{x}{L}-\frac{1}{2}\right)(b-a)+\frac{1}{2}[|x-b|-|x-a|]\right]\\
	%  & = \left\{
	%  \begin{aligned}
	%  &P\frac{x}{L}(b-a), \quad 0\leqslant x \leqslant a,\\
	%  &P\frac{x}{L}(b-a)-x+a, \quad a<x\leqslant b,\\
	%  &P\left(\frac{x}{L}-1\right)(b-a), \quad b<x\leqslant L,\\
	%  \end{aligned}
	%  \right.
	\end{align*}
	Combining Eq (\ref{Eq_erf_lim_1}), (\ref{Eq_erf_lim_2}) and (\ref{Eq_erf_lim_3}), $u_1(x)$ is given by
	\begin{align*}
	u_1(x) &= P\left[\frac{1}{2}\left(\frac{x}{L}-1\right)(b-a)+\frac{1}{2}\frac{x}{L}(b-a)+\frac{1}{2}[(x-b)\sgn(x-b)-(x-a)\sgn(x-a)]\right]\\
	& = P\left[\left(\frac{x}{L}-\frac{1}{2}\right)(b-a)+\frac{1}{2}[(x-b)\sgn(x-b)-(x-a)\sgn(x-a)]\right]\\
	& = P\left[\left(\frac{x}{L}-\frac{1}{2}\right)(b-a)+\frac{1}{2}[|x-b|-|x-a|]\right]\\
	& = \left\{
	\begin{aligned}
	&P\frac{x}{L}(b-a), \quad 0\leqslant x \leqslant a,\\
	&P\frac{x}{L}(b-a)-x+a, \quad a<x\leqslant b,\\
	&P\left(\frac{x}{L}-1\right)(b-a), \quad b<x\leqslant L,\\
	\end{aligned}
	\right.
	\end{align*}
	
	% Note that $\erf(x)\rightarrow 1$ as $x\rightarrow+\infty$ and $\erf(x)\rightarrow -1$ as $x\rightarrow-\infty$, respectively.
	
	Rewriting $u_2(x)$ regarding different domain gives
	\begin{equation*}
	u_2(x) =
	\left\{
	\begin{aligned}
	&P\frac{x}{L}(b-a), \quad 0\leqslant x \leqslant a,\\
	&P\frac{x}{L}(b-a)-x+a, \quad a<x\leqslant b,\\
	&P\left(\frac{x}{L}-1\right)(b-a), \quad b<x\leqslant L.\\
	\end{aligned}
	\right.
	\end{equation*}
	Hence, we conclude that $u_1(x)$ converges to $u_2(x)$ as $\Delta s\rightarrow 0^+$.
\end{proof}

\subsubsection{Finite-Element Method Solutions with Arbitrary Locations of Biological Cells}
\noindent
For the finite-element method, we select the piecewise Lagrangian linear basis functions. We divide the computational domain into $N_e$ mesh elements, with the nodal point $x_1 = 0$ and $x_{N_e+1} = L$. For the implementation, we define the cell density as the count of biological cell in every mesh element divided by the length of the mesh element, hence, it is a constant within every mesh element. In other words, in the mesh element $[x_j, x_{j+1}]$, the count of the biological cell is defined by $$N_c([x_j, x_{j+1}]) = \int_{x_j}^{x_{j+1}}n_c([x_j, x_{j+1}])dx = h n_c([x_j, x_{j+1}]),$$ for any $j\in\{1,\dots,N_e\}$, where $h$ is the size of every mesh element. Different from $(BVP^1_{SPH})$ where $\Delta s$ is the variance of $\delta_{\varepsilon}$, for finite-element methods, we set $\varepsilon = h/3$, such that the integration of $\delta_{h/3}(x-x')$ for any $0<x'<L$ over any mesh element with size $h$, is close to $1$ (see Lemma \ref{lemma_emprical}). With the two approaches, the boundary value problems with Dirichlet boundary condition are defined by
\begin{equation}
\label{Eq_SPH_1d_FEM}
(BVP^2_{SPH})\left\lbrace
\begin{aligned}
-\frac{d^2u_1}{dx^2}&=Ph\sum_{i=1}^{N_s}\delta'_{h/3}(x-s_i), \mbox{$x\in(0,L)$,}\\
u_1(0)&=u_1(L)=0,
\end{aligned}
\right.
\end{equation}
and
\begin{equation}
\label{Eq_density_1d_FEM}
(BVP^2_{den})\left\lbrace
\begin{aligned}
-\frac{d^2u_2}{dx^2}& = Ph\frac{dn_c}{dx}, \mbox{$x\in(0,L)$,}\\
u_2(0)&=u_2(L)=0,
\end{aligned}
\right.
\end{equation}
where $s_i$ is the position of biological cells, $h$ is the mesh size and $N_s$ is the total number of cells in the computational domain. The consistency between $(BVP^2_{SPH})$ and  $(BVP^2_{den})$ can be verified by the following lemma and theorem.
\begin{lemma}
	\label{lemma_emprical}
	{\bf (Empirical rule \cite{Pukelsheim1994})} Given the Gaussian distribution of mean $\mu$ and variance $\varepsilon$:
	\begin{equation*}
	\delta_\varepsilon(x-\mu) = 1/\sqrt{2\pi\varepsilon^2}\exp\{-(x-\mu)^2/(2\varepsilon^2)\},
	\end{equation*}
	then the following integration can be computed:
	\begin{enumerate}
		\item $\int_{\mu-\varepsilon}^{\mu+\varepsilon}\delta_{\varepsilon}(x-\mu)dx\approx0.6827;$
		\item $\int_{\mu-2\varepsilon}^{\mu+2\varepsilon}\delta_{\varepsilon}(x-\mu)dx\approx0.9545;$
		\item $\int_{\mu-3\varepsilon}^{\mu+3\varepsilon}\delta_{\varepsilon}(x-\mu)dx\approx0.9973.$
	\end{enumerate}
\end{lemma}

\begin{theorem}
	\label{Th_1D_FEM}
	Denote  $u^h_1(x)$ and $u^h_2(x)$ respectively the solution to $(BVP^2_{SPH})$ and $(BVP^2_{den})$. With Lagrangian linear basis functions for the finite element method, $u^h_1(x)$ converges to $u^h_2(x)$, as the size of the mesh element $h\rightarrow0^+$, regardless of the positions of biological cells.
\end{theorem}

\begin{proof}
	We define $v^h(x) = u^h_1(x) - u^h_2(x)$, then the boundary value problem to solve $v^h(x)$ is given by
	\begin{equation}
	\label{Eq_1d_v_FEM}
	(BVP^1_v)\left\lbrace
	\begin{aligned}
	-\frac{d^2v^h}{dx^2}& = Ph\left(\sum_{i=1}^{N_s}\delta'_{h/3}(x-s_i)-\frac{dn_c}{dx}\right), \mbox{$x\in(0,L)$,}\\
	v(0)&=v(L)=0.
	\end{aligned}
	\right.
	\end{equation}
	The corresponding Galerkin's form reads as
	\[ (GF_v^1)\left\{
	\begin{aligned}
	&\text{Find $v^h\in H_0^1((0, L))$, such that}\\
	&\int_0^L \frac{dv^h}{dx}\phi'dx = \int_0^LPh\left(\sum_{i=1}^{N_s}\delta'_{h/3}(x-s_i)-\frac{dn_c}{dx}\right)\phi dx,\\
	&\text{for all $\phi\in H_0^1((0, L))$.}
	\end{aligned}
	\right.\]
	
	Using integration by parts and letting $\phi = \phi_j , j\in\{1, \dots, N_e+1\}$, the equation in $(GF_v^1)$ can be rewritten by
	\begin{align*}
	\int_0^L \frac{dv^h}{dx}\phi'_jdx &= \int_0^LPh\left(\sum_{i=1}^{N_s}\delta'_{h/3}(x-s_i)-\frac{dn_c}{dx}\right)\phi_j dx
	\\&= \left[Ph\sum_{i=1}^{N_s}\delta_{h/3}(x-s_i)\phi_j\right]^L_0 - [Phn_c\phi_j]^L_0 - \int_{0}^L Ph\left(\sum_{i=1}^{N_s}\delta_{h/3}(x-s_i) - n_c\right)\phi'_jdx \\
	\mbox{(Boundary condition)} & = - \int_{0}^L Ph\left(\sum_{i=1}^{N_s}\delta_{h/3}(x-s_i) - n_c\right)\phi'_jdx \\
	& = Ph\sum_{j=1}^{N_e} \left\{\int_{x_j}^{x_{j+1}}n_c \phi_j' dx -\int_{x_j}^{x_{j+1}}\sum_{i=1}^{N_s}\delta_{h/3}(x-s_i)\phi'_j dx\right\}\\
	& =  P\sum_{j=1}^{N_e} \left[N_c([x_j,x_{j+1}]) - \int_{x_j}^{x_{j+1}} \sum_{i=1}^{N_s}\delta_{h/3}(x-s_i)dx \right] \\
	%\mbox{\scriptsize ($h\rightarrow0^+$)}& \rightarrow P\sum_{j=1}^{N_e} \{[N_c([x_j,x_{j+1}]) - \int_{x_j}^{x_{j+1}} \sum_{i=1}^{N_s}\delta(x-s_i)dx ] \} \\
	\mbox{($h\rightarrow0^+$, Lemma \ref{lemma_emprical})}& \rightarrow 0,
	\end{align*}
	since it can be defined that $\displaystyle N([x_j,x_{j+1}]) = \int_{x_j}^{x_{j+1}} \sum_{i=1}^{N_s}\delta(x-s_i)dx$.
\end{proof}

\section{Mathematical Models in Two Dimensions}\label{MathsModel_2D}
\subsection{Smoothed Particle Approach and Cell Density Approach}
\noindent
In multi dimensional case, the equation of conservation of momentum over the computational domain $\Omega$, without considering inertia, is given by $$-\nabla\cdot\boldsymbol{\sigma}=\boldsymbol{f}.$$ We consider a linear, homogeneous and isotropic domain, with Hooke's Law, the stress tensor $\boldsymbol{\sigma}$ is defined as
\begin{equation}
\label{Eq_sigma_2D}
\boldsymbol{\sigma} = \frac{E}{1+\nu}\left\lbrace\boldsymbol{\epsilon}+\tr(\epsilon)\left[\frac{\nu}{1-2\nu}\right]\boldsymbol{I}\right\rbrace,
\end{equation}
where $E$ is the Young's modulus of the material, $\nu$ is Poisson's ratio and $\boldsymbol{\epsilon}$ is the infinitesimal strain tensor: $$\boldsymbol{\epsilon} = \frac{1}{2}[\boldsymbol{\nabla u}+(\boldsymbol{\nabla u})^T].$$
Considering a subdomain $\Omega_w\subset\Omega$, where the center positions of the biological cells are located, then the SPH approach and cell density approach with homogeneous Dirichlet boundary condition are derived by
\begin{equation}
\label{Eq_SPH_2D}
(BVP^3_{SPH})\left\lbrace
\begin{aligned}
-\nabla\cdot\boldsymbol{\sigma}&= P_{SPH}\sum_{i=1}^{N_s}\nabla\delta_{\varepsilon}(\boldsymbol{x}-\boldsymbol{s_i}), \mbox{$\boldsymbol{x}\in\Omega$,}\\
\boldsymbol{u_1}(\boldsymbol{x}) &= \boldsymbol{0}, \mbox{$\boldsymbol{x}\in\partial\Omega$,}
\end{aligned}
\right.
\end{equation}
and
\begin{equation}
\label{Eq_density_2D}
(BVP^3_{den})\left\lbrace
\begin{aligned}
-\nabla\cdot\boldsymbol{\sigma}&= P_{den}\nabla\cdot (n_c\boldsymbol{I}), \mbox{$\boldsymbol{x}\in\Omega$,}\\
\boldsymbol{u_2}(\boldsymbol{x}) &= \boldsymbol{0}, \mbox{$\boldsymbol{x}\in\partial\Omega$.}
\end{aligned}
\right.
\end{equation}

\subsection{Consistency between Two Approaches in Finite-Element Method}
\noindent
To prove the consistency between these two approaches, we define that for the triangular mesh element $e_k, k\in\{1,\dots,N_e\}$, where $N_e$ is the total number of mesh elements in $\Omega$,  the density of biological cell $n_c(e_k)$ is constant and the count of biological cells $N_c(e_k)$ is expressed by $$N_c(e_k) = \int_{e_k}n_c(e_k)d\Omega = A(e_k)n_c(e_k),$$ where $A(e_k)$ is the area of mesh element $e_k$. Similar to the process of proof in one dimension, we state the following theorem:
\begin{theorem}
	\label{Th_2D_FEM}
	Denote  $\boldsymbol{u^h_1}(\boldsymbol{x})$ and $\boldsymbol{u^h_2}(\boldsymbol{x})$ respectively the solution to $(BVP^3_{SPH})$ and $(BVP^3_{den})$ with $P_{SPH}=P_{den}=P$. With Lagrange linear basis functions for the finite element method, $\boldsymbol{u^h_1}(\boldsymbol{x})$ converges to $\boldsymbol{u^h_2}(\boldsymbol{x})$, as the size of a triangular mesh element is taken to zero, regardless of the positions of biological cells.
\end{theorem}

\begin{proof}
	We consider $\boldsymbol{v^h}(\boldsymbol{x}) = \boldsymbol{u^h_1}(\boldsymbol{x}) - \boldsymbol{u^h_2}(\boldsymbol{x})$, then the boundary value problem to solve $\boldsymbol{v^h}(\boldsymbol{x})$ is given by
	\begin{equation}
	\label{Eq_2d_v_FEM}
	(BVP^3_v)\left\lbrace
	\begin{aligned}
	-\nabla\cdot\boldsymbol{\sigma} &= P\left[\sum_{i=1}^{N_s}\nabla\delta_{\varepsilon}(\boldsymbol{x}-\boldsymbol{s_i}) - \nabla\cdot (n_c\boldsymbol{I})\right], \mbox{$\boldsymbol{x}\in\Omega$,}\\
	\boldsymbol{v}(\boldsymbol{x}) &= \boldsymbol{0}, \mbox{$\boldsymbol{x}\in\partial\Omega$}.
	\end{aligned}
	\right.
	\end{equation}
	The corresponding Galerkin's form reads as
	\[ (GF^3_v)\left\{
	\begin{aligned}
	&\text{Find $\boldsymbol{v^h}\in \boldsymbol{H_0^1}(\Omega)$, such that}\\
	&\int_\Omega \boldsymbol{\sigma}(\boldsymbol{v^h}):\nabla\boldsymbol{\phi^h} d\Omega = \int_\Omega P\left[\sum_{i=1}^{N_s}\nabla\delta_{h}(\boldsymbol{x}-\boldsymbol{s_i}) - \nabla\cdot (n_c\boldsymbol{I})\right]\cdot\boldsymbol{\phi^h} d\Omega,\\
	&\text{for all $\boldsymbol{\phi^h}\in \boldsymbol{H_0^1}(\Omega)$.}
	\end{aligned}
	\right.\]
	
	With integral by parts and letting $\boldsymbol{\phi^h} = \boldsymbol{\phi^h_k} , k\in\{1, \dots, N_e\}$, the equation in $(GF_v^3)$ can be rewritten by
	\begin{align*}
	\int_\Omega \boldsymbol{\sigma}(\boldsymbol{v^h})\nabla\boldsymbol{\phi^h_k} d\Omega &= \int_\Omega P\left[\sum_{i=1}^{N_s}\nabla\delta_{\varepsilon}(\boldsymbol{x}-\boldsymbol{s_i}) - \nabla\cdot (n_c\boldsymbol{I})\right]\cdot\boldsymbol{\phi^h_k} d\Omega \\
	& = P\left\lbrace\left[\int_{\partial\Omega} \sum_{i=1}^{N_s}\delta_{\varepsilon}(\boldsymbol{x}-\boldsymbol{s_i})\boldsymbol{\phi^h_k}\boldsymbol{n}d\Gamma-\int_\Omega \sum_{i=1}^{N_s}\delta_{\varepsilon}(\boldsymbol{x}-\boldsymbol{s_i})\nabla\cdot\boldsymbol{\phi^h_k}d\Omega \right] \right. \\
	&\left. -\left[\int_{\partial\Omega}n_c(\boldsymbol{I\phi^h_k})\cdot\boldsymbol{n}d\Gamma - \int_\Omega n_c\boldsymbol{I}:\nabla\boldsymbol{\phi^h_k}d\Omega\right]\right\rbrace\\
	\mbox{(Boundary condition)} &= -P\int_\Omega \sum_{i=1}^{N_s}\delta_{\varepsilon}(\boldsymbol{x}-\boldsymbol{s_i})\nabla\cdot\boldsymbol{\phi^h_k} - n_c\boldsymbol{I}:\nabla\boldsymbol{\phi^h_k}d\Omega \\
	& = - P\sum_{k=1}^{N_e}\int_{e_k} \sum_{i=1}^{N_s}\delta_{\varepsilon}(\boldsymbol{x}-\boldsymbol{s_i})\nabla\cdot\boldsymbol{\phi^h_k} - n_c\boldsymbol{I}:\nabla\boldsymbol{\phi^h_k}d\Omega \\
	& = -P\sum_{k=1}^{N_e}\nabla\cdot\boldsymbol{\phi^h_k}\int_{e_k} \sum_{i=1}^{N_s}\delta_{\varepsilon}(\boldsymbol{x}-\boldsymbol{s_i}) - n_c d\Omega\\
	& = P\sum_{k=1}^{N_e}\nabla\cdot\boldsymbol{\phi^h_k}\left[N_c(e_k) - \int_{e_k} \sum_{i=1}^{N_s}\delta_{\varepsilon}(\boldsymbol{x}-\boldsymbol{s_i})d\Omega \right]\\
	\mbox{($\varepsilon\rightarrow0^+$)}& \rightarrow 0,
	\end{align*}
	since it can be defined that $\displaystyle N_c(e_k) = \int_{e_k} \sum_{i=1}^{N_s}\delta(\boldsymbol{x}-\boldsymbol{s_i})d\Omega$. Note that in two dimensions, $\varepsilon$ needs to be sufficiently small compared to the size of a triangular mesh element.
	
\end{proof}

\section{Results}\label{Results}
\noindent
Results in both one and two dimensions are discussed in this section. Since the objective of this manuscript is to investigate the consistency and the connections between the SPH approach and the cell density approach, all the parameters are dimensionless.

\subsection{One-Dimensional Results}
\noindent
We show the results by analytical solutions in Figure \ref{Fig_1d_linspace_exact} with various values of $\Delta s$ (i.e. depending on different number of biological cells in the subdomain $(a, b)$). Here, the computational domain is $(0,7)$ with $L=7$ and the subdomain where the biological cells locate uniformly is $(2,5)$ with $a=2$ and $b=5$. With the decrease of the variance in the Gaussian distribution in $(BVP^2_{SPH})$, the curves gradually overlap, which verifies the convergence between the analytical solutions to these two approaches.
\begin{figure}[htpt]
	\centering
	\subfigure[$\Delta s = 0.3$]{
	\includegraphics[width=0.45\textwidth]{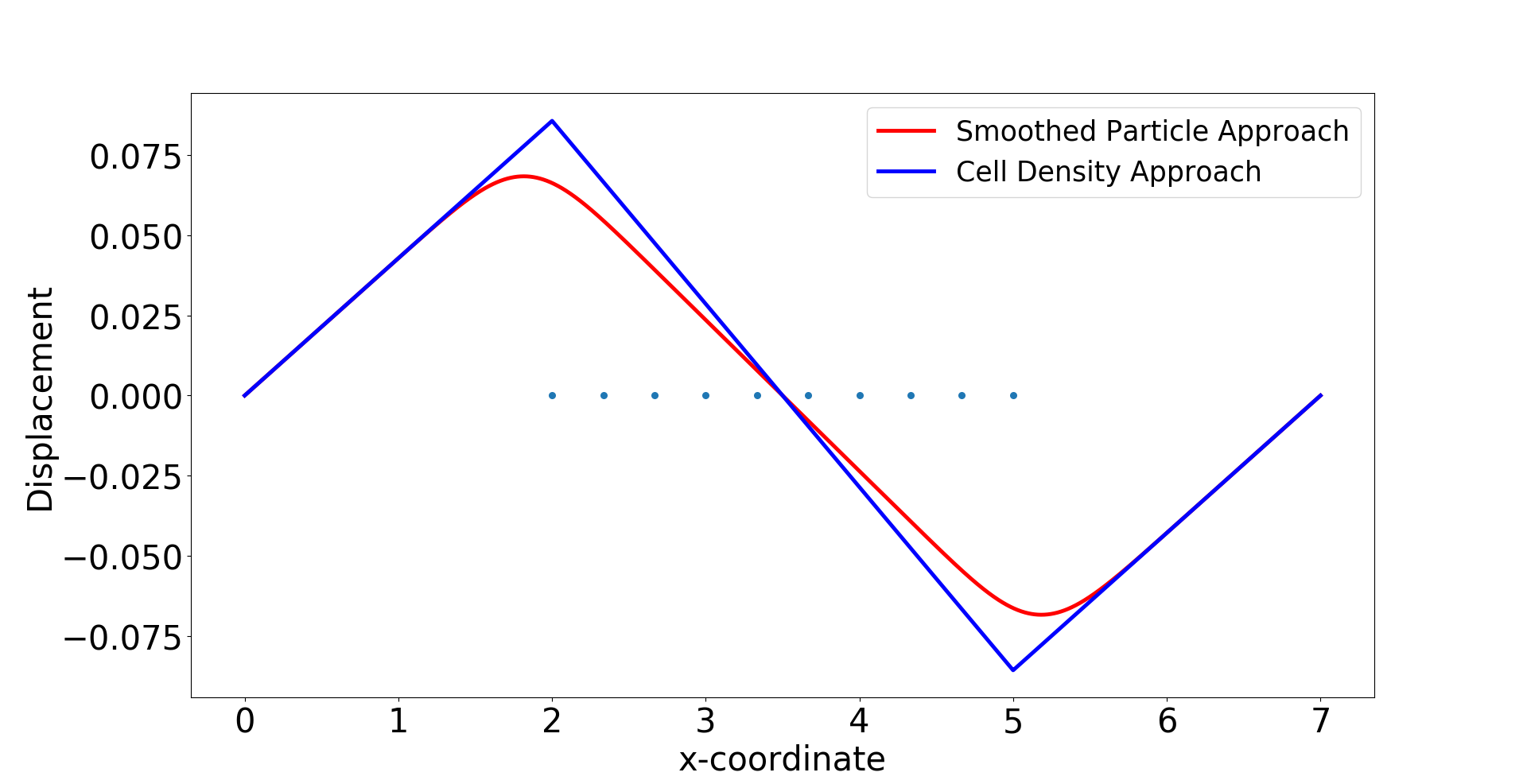}}
	\subfigure[$\Delta s = 0.06$]{
	\includegraphics[width=0.45\textwidth]{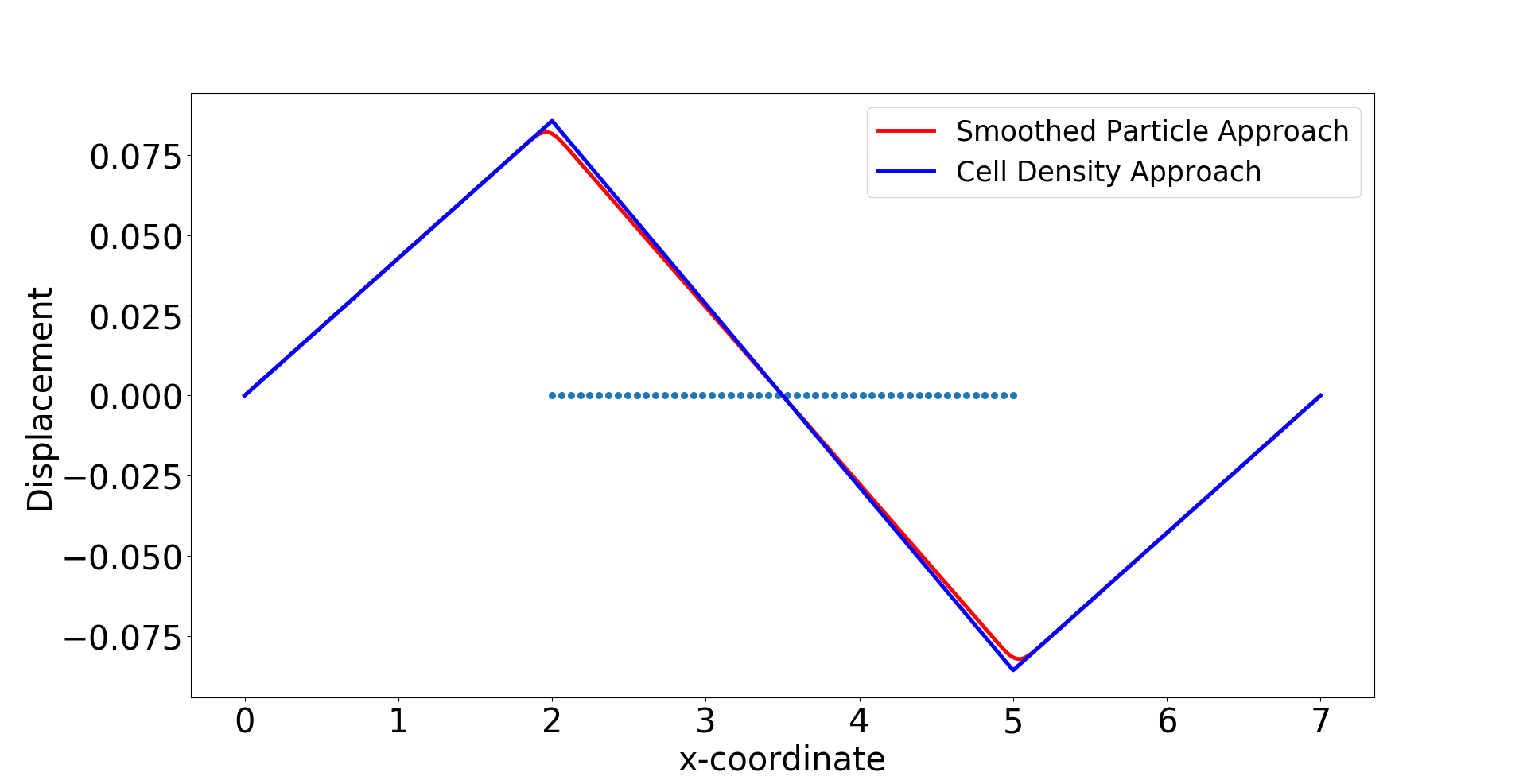}}
	\subfigure[$\Delta s = 0.03$]{
	\includegraphics[width=0.45\textwidth]{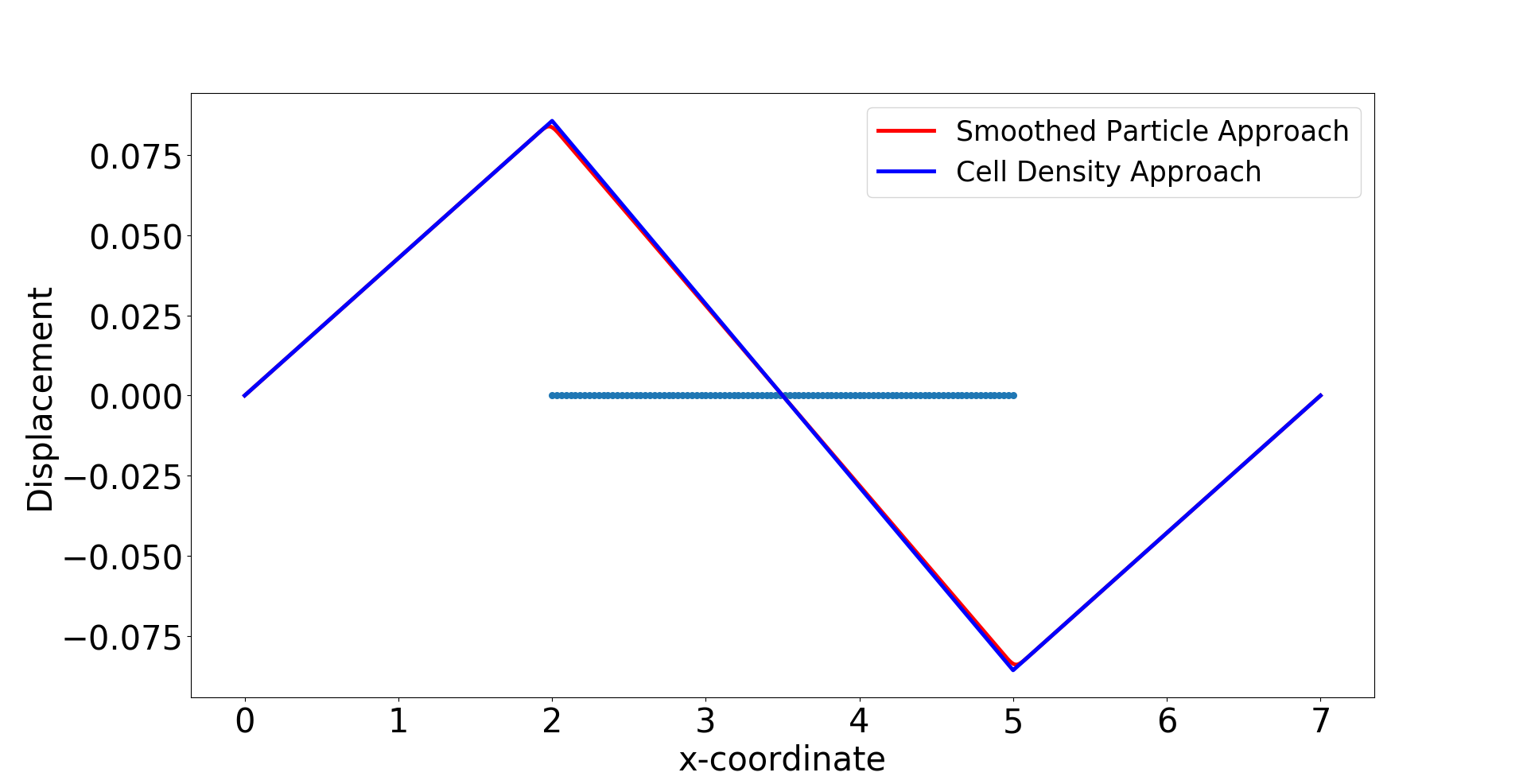}}
	\subfigure[$\Delta s = 0.006$]{
	\includegraphics[width=0.45\textwidth]{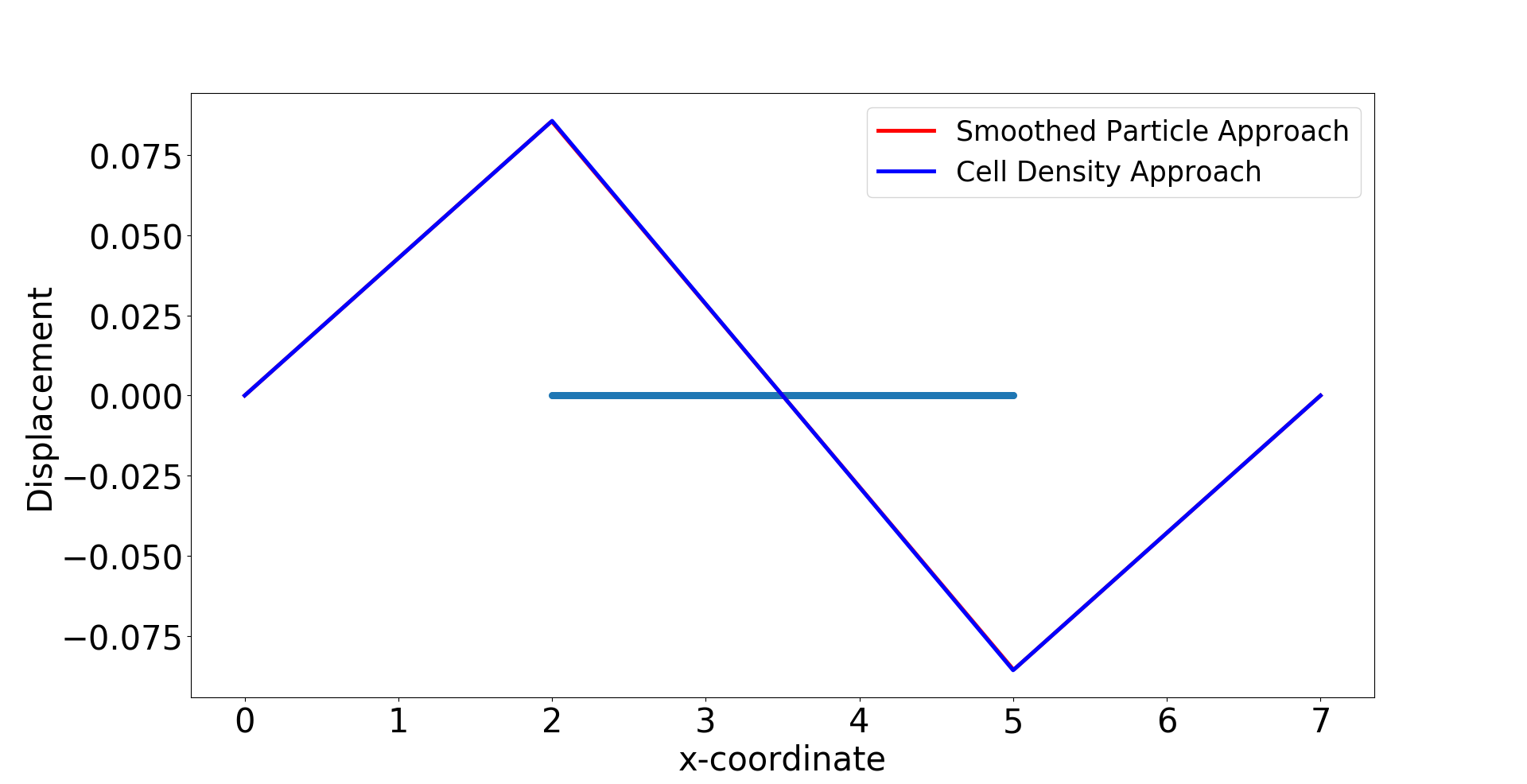}}
\caption{The exact solutions to $(BVP^1_{SPH})$ and $(BVP^1_{den})$ are shown, with various values of $\Delta s$, which is the distance between centre positions of any two adjacent biological cells. Blue points are the centre positions of biological cells. Red curves represent the solutions to $(BVP^1_{SPH})$ and blue curves represent the solutions to $(BVP^1_{den})$.}
\label{Fig_1d_linspace_exact}
\end{figure}

To implement the model, there are two different algorithms shown in Figure \ref{Diag_cell_density} and \ref{Diag_cell_positions}. Depending on different circumstances, the implementation method is elected. The cell density in one dimension is defined as the number of cells per length unit. In other words, the cell count in a given domain can be computed by integrating the cell density over the domain. If the cell density function can be expressed analytically and the first order derivative of the function exists, then a certain bin length $d$ is chosen and the cell count in every bin of $d$ length is calculated. Then we generalize the center positions of cells in every bin of length $d$, thus, the SPH approach can be implemented, as it is indicated in Figure \ref{Diag_cell_density}. However, it is not always straightforward to obtain the analytical expression of cell density. If the center positions of cells are given, the number of cells in each mesh element can be counted, hence, the cell density will be computed analogously at each mesh points. Therefore, the boundary value problem of cell density approach is solved by numerical methods, for example, the finite-element methods.
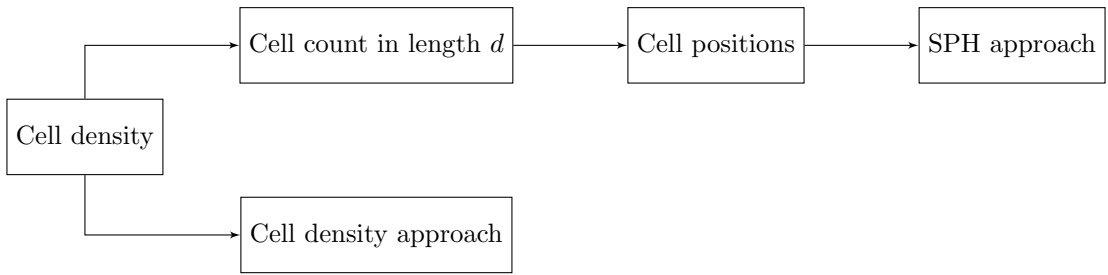
\begin{figure}[htpb]
	\centering
\begin{tikzpicture}\footnotesize
\node[block] (a) {Cell density};
\node[block, above right = 0.2cm and 1cm of a] (b) {Cell count in length $d$};
\node[block, below =1.5cm of b]   (c){Cell density approach};
\node[block, right =1.5cm of b]   (d){Cell positions};
\node[block, right =1.5cm of d]   (e){SPH approach};

\draw[line] (a.north) |- (b.west);
\draw[line] (a.south) |- (c.west);
\draw[line] (b.east) -- (d.west);
\draw[line] (d.east) -- (e.west);
%\draw[line] ([xshift=-1cm]b.south) -- ([xshift=-1cm]c.north);
%\draw[line] ([xshift=1cm]c.north) -- ([xshift=1cm]b.south);
%\draw[] (b.east) -- ++(10pt,0) coordinate[yshift=-1.7cm](l){} |- (c.east);
%\draw[<->,>=latex'] (d.west) -- ++(-10pt,0) coordinate[yshift=-1.7cm,](r){} |- (e.west);
%\draw[-] ([xshift=1cm]c.north) -- ([xshift=1cm]b.south);
%\draw[line] (l) -- (r);
\end{tikzpicture}
\caption{With exact expression of cell density function and the first order derivative of the function exists, cell density approach is implemented directly. Based on the cell density, the number of cells in a certain region with length $d$ is determined and subsequently, the center positions of cells can be generalized. Hence, the SPH approach is implemented. }
\label{Diag_cell_density}
\end{figure}

\begin{figure}[htpb]
	\centering
\begin{tikzpicture}\footnotesize
\node[block] (a) {Cell positions};
\node[block, above right = 0.2cm and 1cm of a] (b) {Cell count in mesh element};
\node[block, below =1.5cm of b]   (c){SPH approach};
\node[block, right =1.5cm of b]   (d){Cell density};
\node[block, right =1.5cm of d]   (e){Cell density approach};

\draw[line] (a.north) |- (b.west);
\draw[line] (a.south) |- (c.west);
\draw[line] (b.east) -- (d.west);
\draw[line] (d.east) -- (e.west);
%\draw[line] ([xshift=-1cm]b.south) -- ([xshift=-1cm]c.north);
%\draw[line] ([xshift=1cm]c.north) -- ([xshift=1cm]b.south);
%\draw[] (b.east) -- ++(10pt,0) coordinate[yshift=-1.7cm](l){} |- (c.east);
%\draw[<->,>=latex'] (d.west) -- ++(-10pt,0) coordinate[yshift=-1.7cm,](r){} |- (e.west);
%\draw[-] ([xshift=1cm]c.north) -- ([xshift=1cm]b.south);
%\draw[line] (l) -- (r);
\end{tikzpicture}
\caption{Given the center positions of cells, one can directly implement the SPH model. Computing the number of cells in every mesh element and divided by the length of the mesh element results into the cell density. Subsequently, cell density approach can be implemented.}
\label{Diag_cell_positions}	
\end{figure}
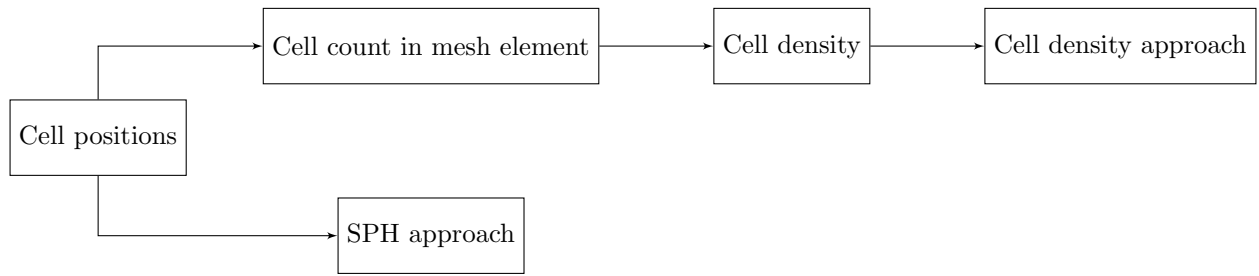

In this manuscript, all the numerical results are derived by finite-element methods with Lagrangian linear basis functions. Regarding the first implementation method (see Figure \ref{Diag_cell_density}), we show the results with a Gaussian distribution and sine function as cell density functions; see Figure \ref{Fig_cell_den_Gauss} and \ref{Fig_cell_den_sine}. We start with the simulations in which we keep the number of cells and the center positions of the cells the same, then we refine the mesh. In Figure \ref{Fig_cell_den_sine}\subref{Fig_cell_den_sine_a}-\subref{Fig_cell_den_sine_c}, the bin length $d$ is $0.35$, and the mesh size is a function of $d$. The results solved by SPH approach become smoother. With various values of $d$, the solutions to the approaches are overlapping only when the factor between the $d$ and mesh size is closer to $1$. From Figure \ref{Fig_cell_den_sine}\subref{Fig_cell_den_sine_d} to \subref{Fig_cell_den_sine_f}, the mesh is fixed and we vary the value of $d$. We note that in Figure \ref{Fig_cell_den_sine}\subref{Fig_cell_den_sine_f}, the solution to the SPH approach is significantly different from the solution to the cell density approach. It is mainly caused by the fact that $d$ is too small and there is barely any fluctuation with the count of cells in every $d$ length subdomain, while with the Gaussian distribution as the cell density function, the majority of the cells are centered around $x=3.5$. Hence, the solution to SPH approach still manages to be comparable with the solution to the cell density approach; see Figure \ref{Fig_cell_den_Gauss}\subref{Fig_cell_den_Gauss_f}. Numerical results of the simulation in Figure \ref{Fig_cell_den_Gauss} are displayed in Table \ref{Tbl_1D_Gaussian}. There are some noticeable differences between two approaches, in particular the convergence rate in the $H^1$-norm: thanks to the given, differentiable cell density function, the cell density approach converges faster. In addition, the cell density approach requires less computational time with a factor of $15$.
\begin{figure}[htpb]
	\centering
	\subfigure[$h(d) = d/5, h = 0.07, d=0.35$]{
		\label{Fig_cell_den_Gauss_a}
		\includegraphics[width=0.3\textwidth]{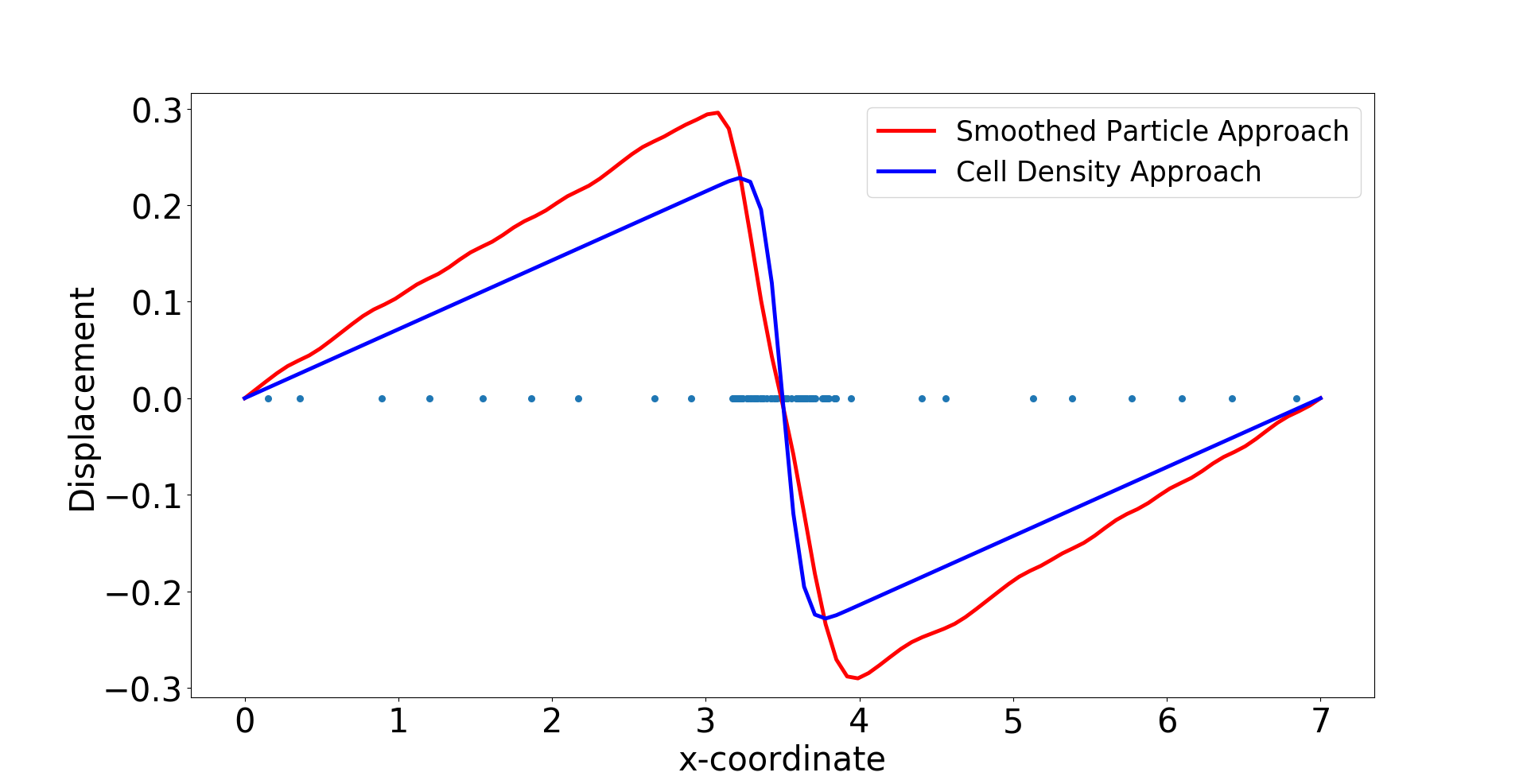}}
	\subfigure[$h(d) = d/10, h = 0.035, d=0.35$]{
		\label{Fig_cell_den_Gauss_b}
		\includegraphics[width=0.3\textwidth]{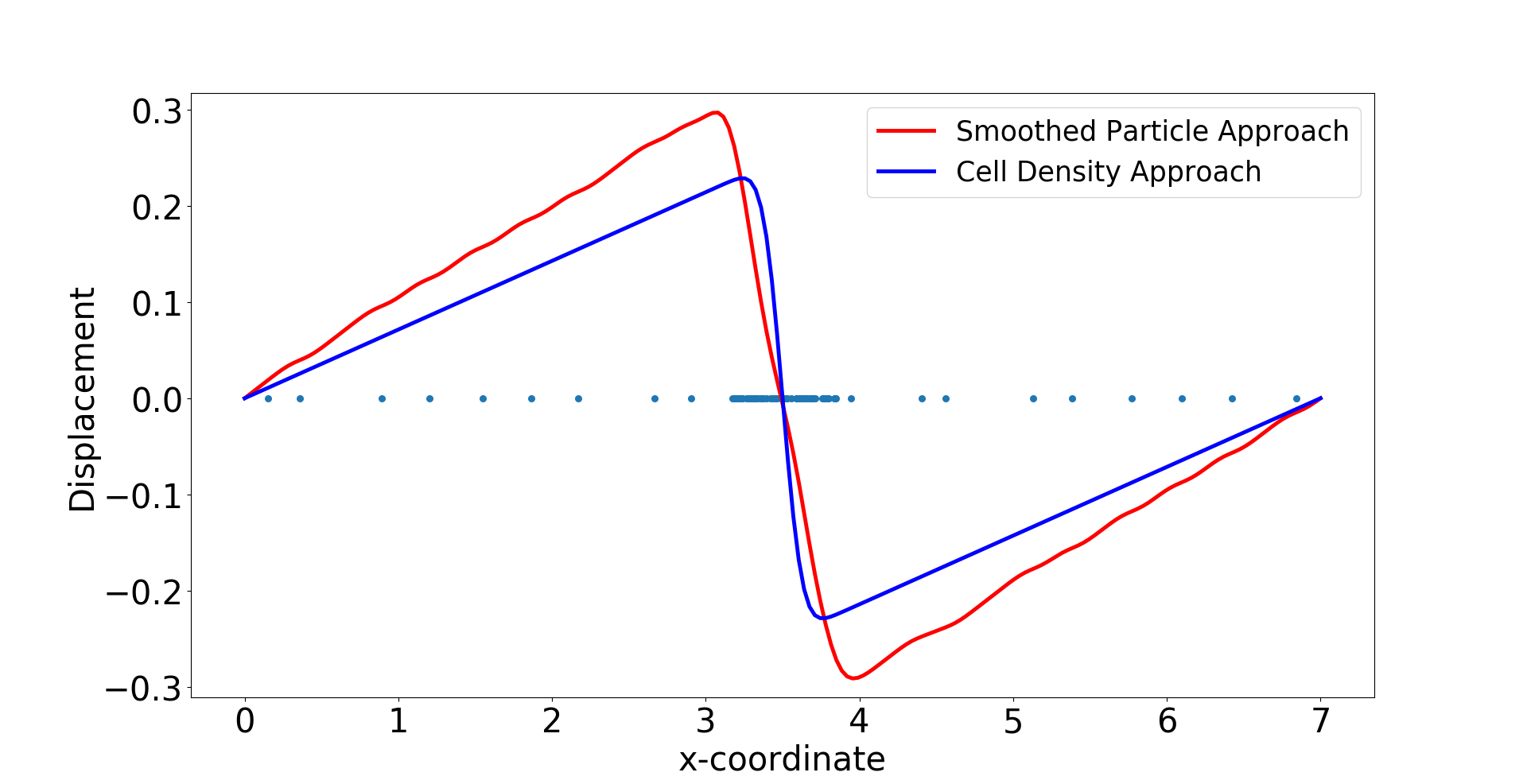}}
%	\subfigure[$h = d/25, h = 0.014$]{
%		\includegraphics[width=0.3\textwidth]{Cell_den_alg/Figure_1_Gaussian_20_500_88.png}}
	\subfigure[$h(d) = d/50, h = 0.007, d=0.35$]{
		\label{Fig_cell_den_Gauss_c}
		\includegraphics[width=0.3\textwidth]{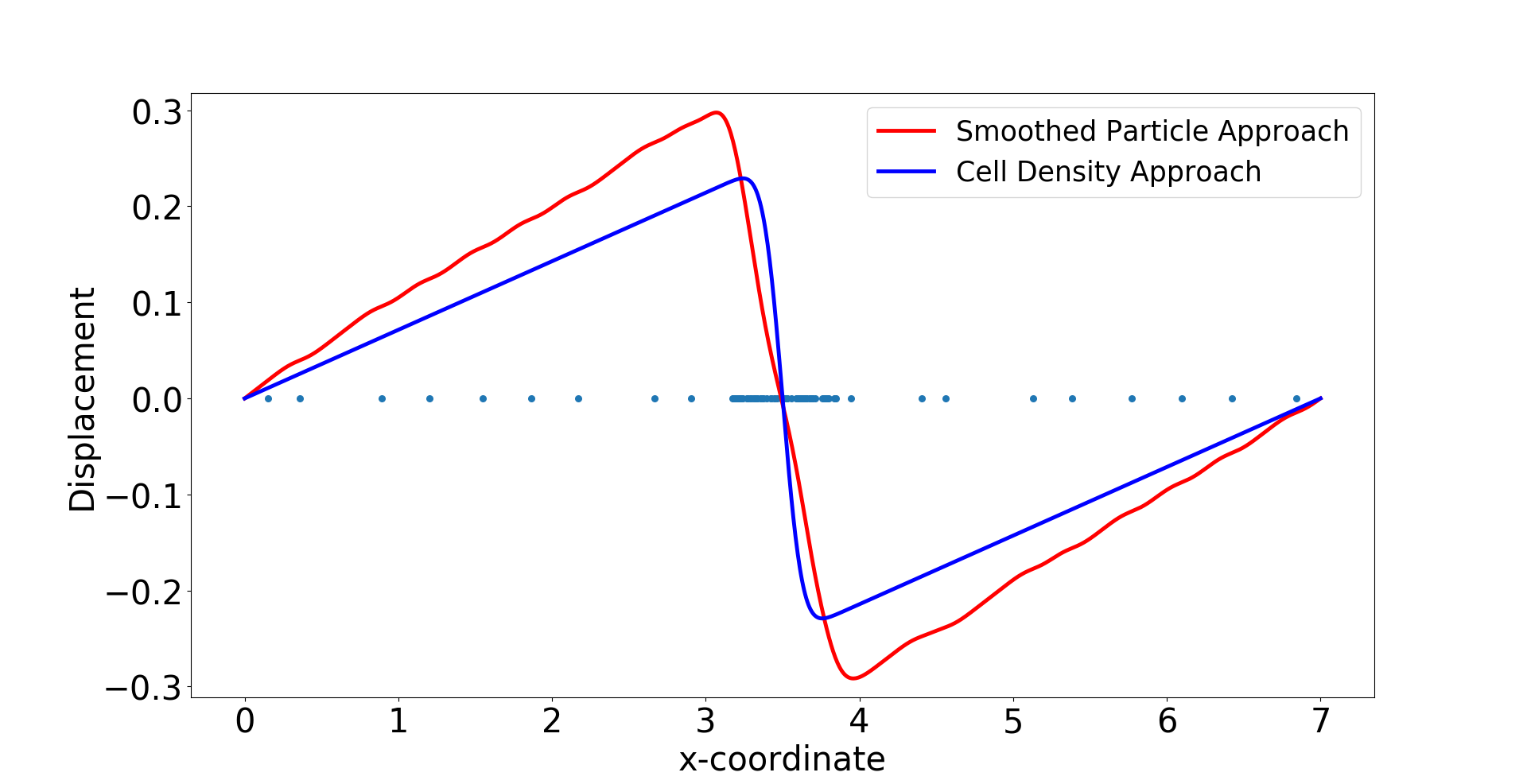}}
	\subfigure[$d(h) = 2h, h = 0.07,d=0.14$]{
		\label{Fig_cell_den_Gauss_d}
		\includegraphics[width=0.3\textwidth]{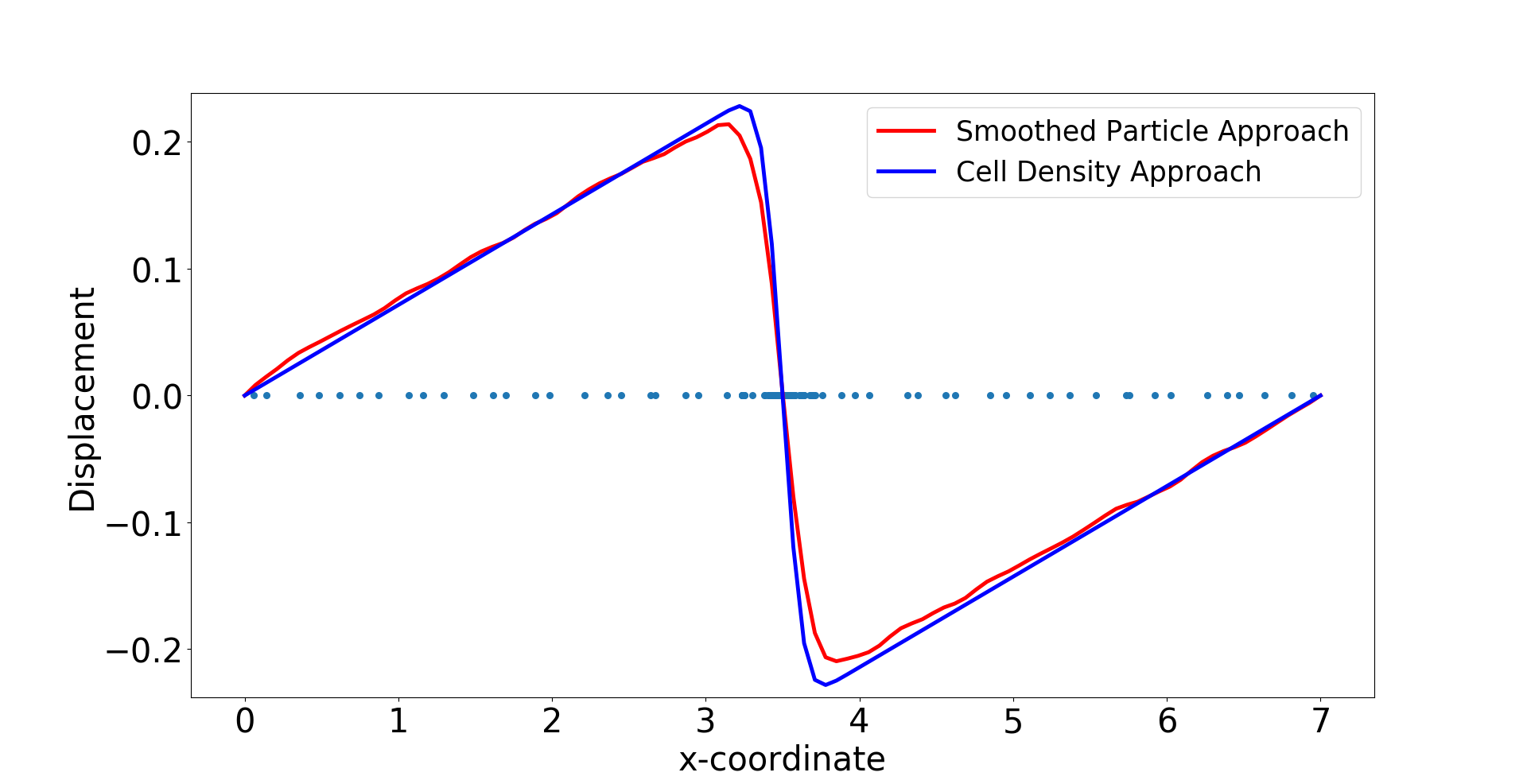}}
	\subfigure[$d(h) = h, h = 0.07,d=0.07$]{
		\label{Fig_cell_den_Gauss_e}
		\includegraphics[width=0.3\textwidth]{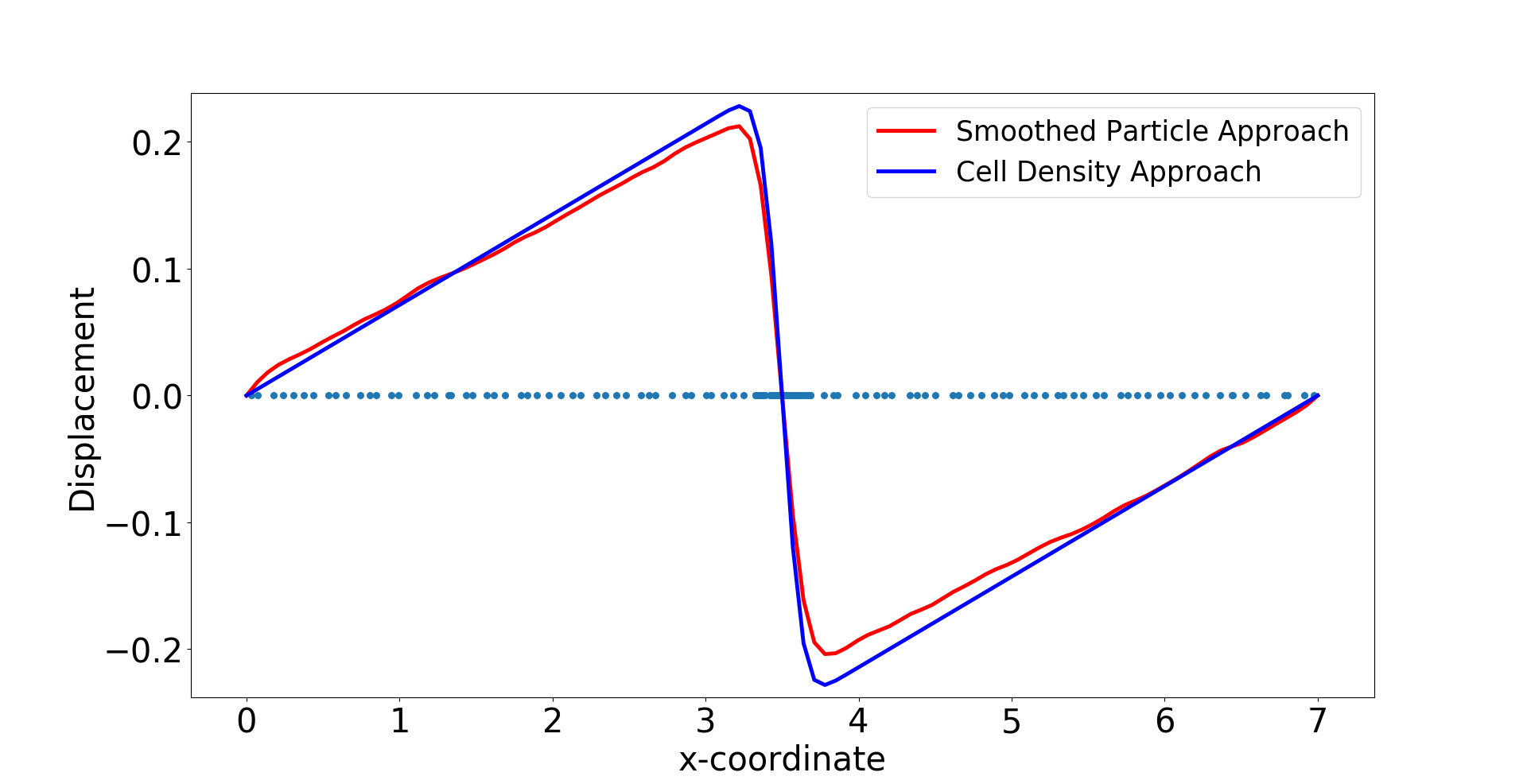}}
	\subfigure[$d(h) = h/4, h = 0.07,d=0.0175$]{
		\label{Fig_cell_den_Gauss_f}
		\includegraphics[width=0.3\textwidth]{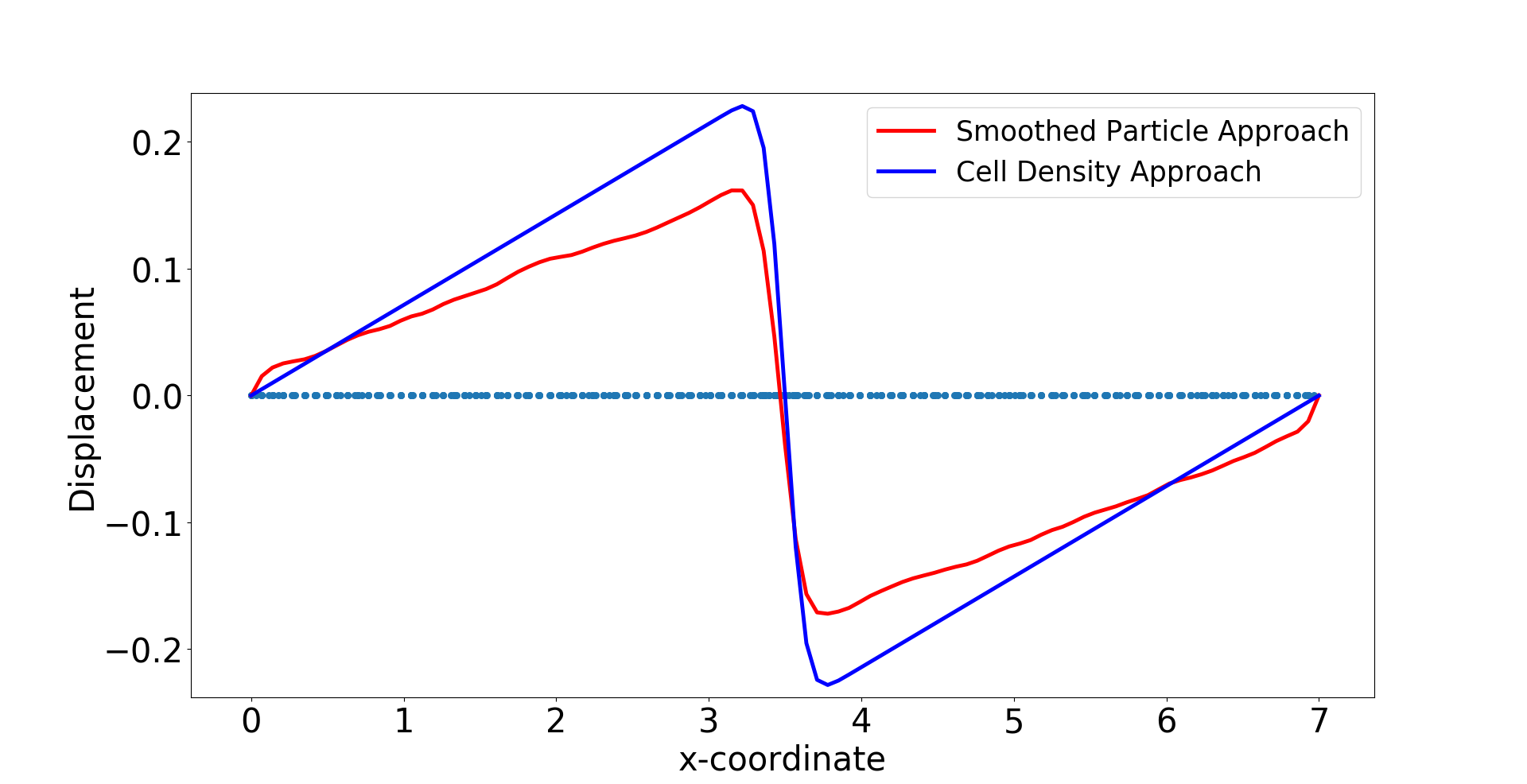}}
\caption{The cell density function is Gaussian distribution and using the algorithm in Figure \ref{Diag_cell_density}, different simulations are carried out with various mesh size and the total number of cells. Blue curves represent the solutions to $(BVP^2_{SPH})$, and ref curves are the solutions to $(BVP^2_{den})$ with $n_c(x) = 50\times 1/\sqrt{2\pi\times0.1^2}\exp\{-(x-3.5)^2/(2\times 0.1^2)\}$. In Subfigure \subref{Fig_cell_den_Gauss_a}--\subref{Fig_cell_den_Gauss_c}, we set $d = 0.35$ and cell positions are fixed, as $h$ is decreasing. From Subfigure \subref{Fig_cell_den_Gauss_d} to Subfigure \subref{Fig_cell_den_Gauss_f}, we use the same finite-element method settings (where $h$ is sufficiently small with $h = 0.07$), and simulations are carried out with various values of $d$.}	
\label{Fig_cell_den_Gauss}
\end{figure}

\begin{table}[htpb]
	\centering
	\caption{Numerical results of two approaches in one dimension, where the cell density function is Gaussian distribution: $n_c(x) = 50\times 1/\sqrt{2\pi\times0.1^2}\exp\{-(x-3.5)^2/(2\times 0.1^2)\}$. Here, we define $N_s=88$ and the mesh size $h = 0.07$. The results are solved by finite-element method with algorithm in Figure \ref{Diag_cell_density}. }
	\begin{tabular}{m{6cm}<{\centering}m{5cm}<{\centering}m{5cm}<{\centering}}
		\toprule
		& {\bf SPH Approach} & {\bf Cell Density Approach}\\
		\midrule
		$\|u\|_{L^2((0,L))}$ & $0.5441481069175041$ & $0.36197930815501245$\\
		$\|u\|_{H^1(((0,L))}$ & $0.9642151731656272$ & $0.871720645462775$\\
		Convergence rate of $L^2-norm$ & $1.75281178$ & $1.826378221$\\
		Convergence rate of $H^1-norm$ & $0.66372231$ & $1.816659924$\\
		Reduction ratio of the subdomain $(a,b)$ (\%) & $13.88062$ & $9.52381$ \\
		Time cost $(s)$ & $0.045070$ & $0.0032084$\\ 
		\bottomrule
	\end{tabular}
	\label{Tbl_1D_Gaussian}
\end{table}

\begin{figure}[htpb]
	\centering
	\subfigure[$h(d) = d/5, h = 0.07, d=0.35$]{
		\label{Fig_cell_den_sine_a}
		\includegraphics[width=0.3\textwidth]{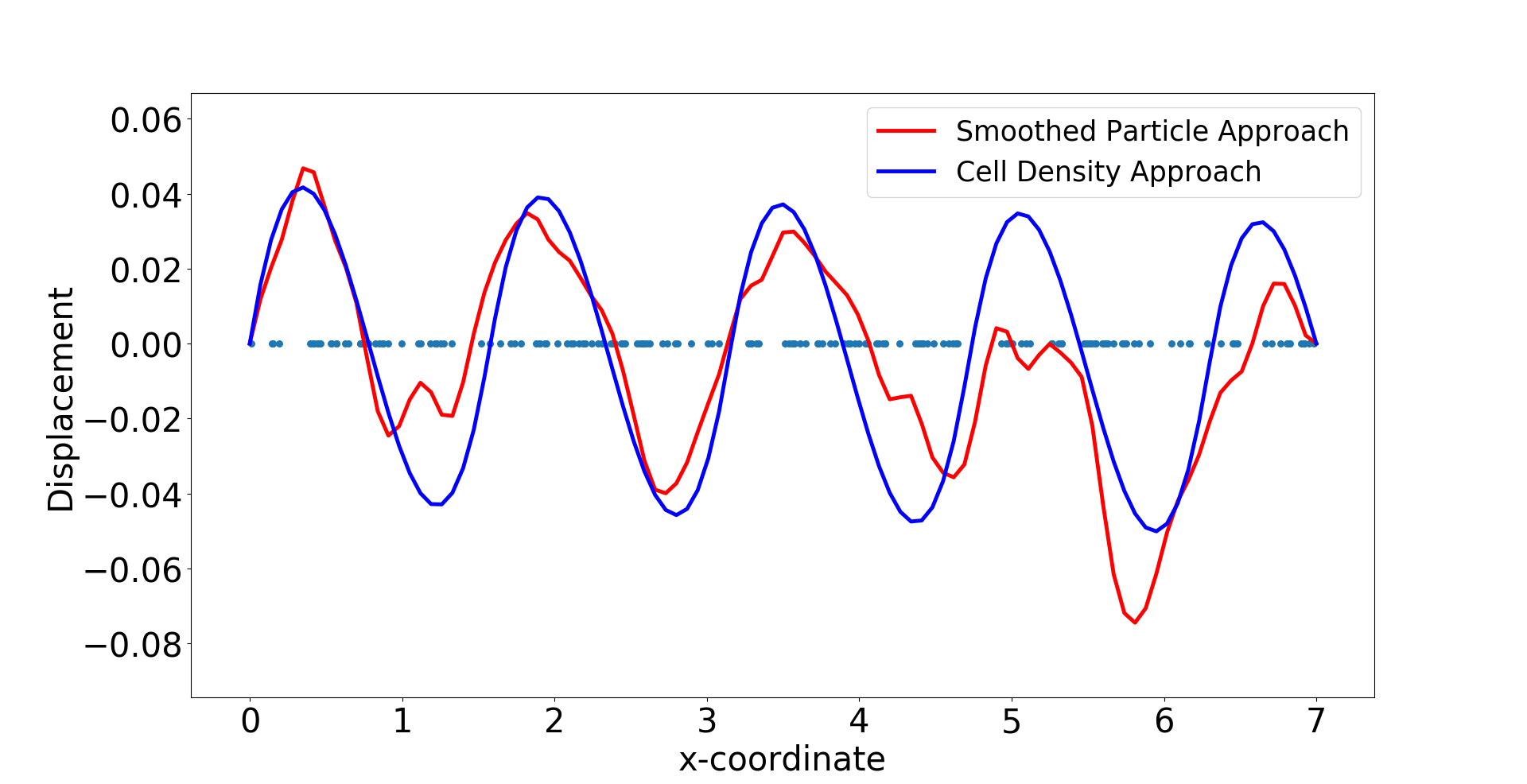}}
	\subfigure[$h(d) = d/10, h = 0.035, d=0.35$]{
		\label{Fig_cell_den_sine_b}
		\includegraphics[width=0.3\textwidth]{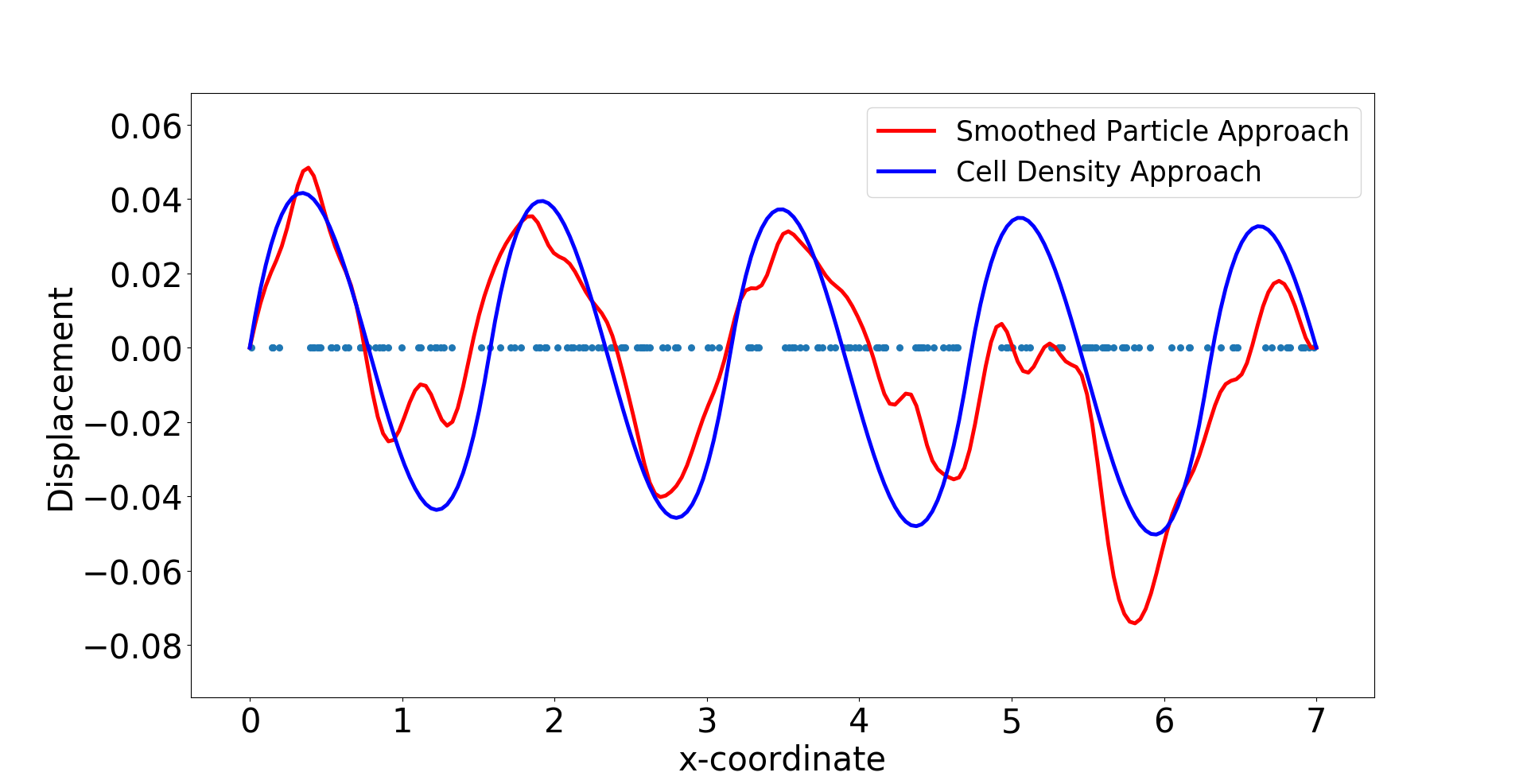}}
%	\subfigure[$h = d/25, h = 0.014$]{
%		\includegraphics[width=0.3\textwidth]{Cell_den_alg/Figure_1_sine_20_500_184.png}}
	\subfigure[$h(d) = d/50, h = 0.007, d=0.35$]{
		\label{Fig_cell_den_sine_c}
		\includegraphics[width=0.3\textwidth]{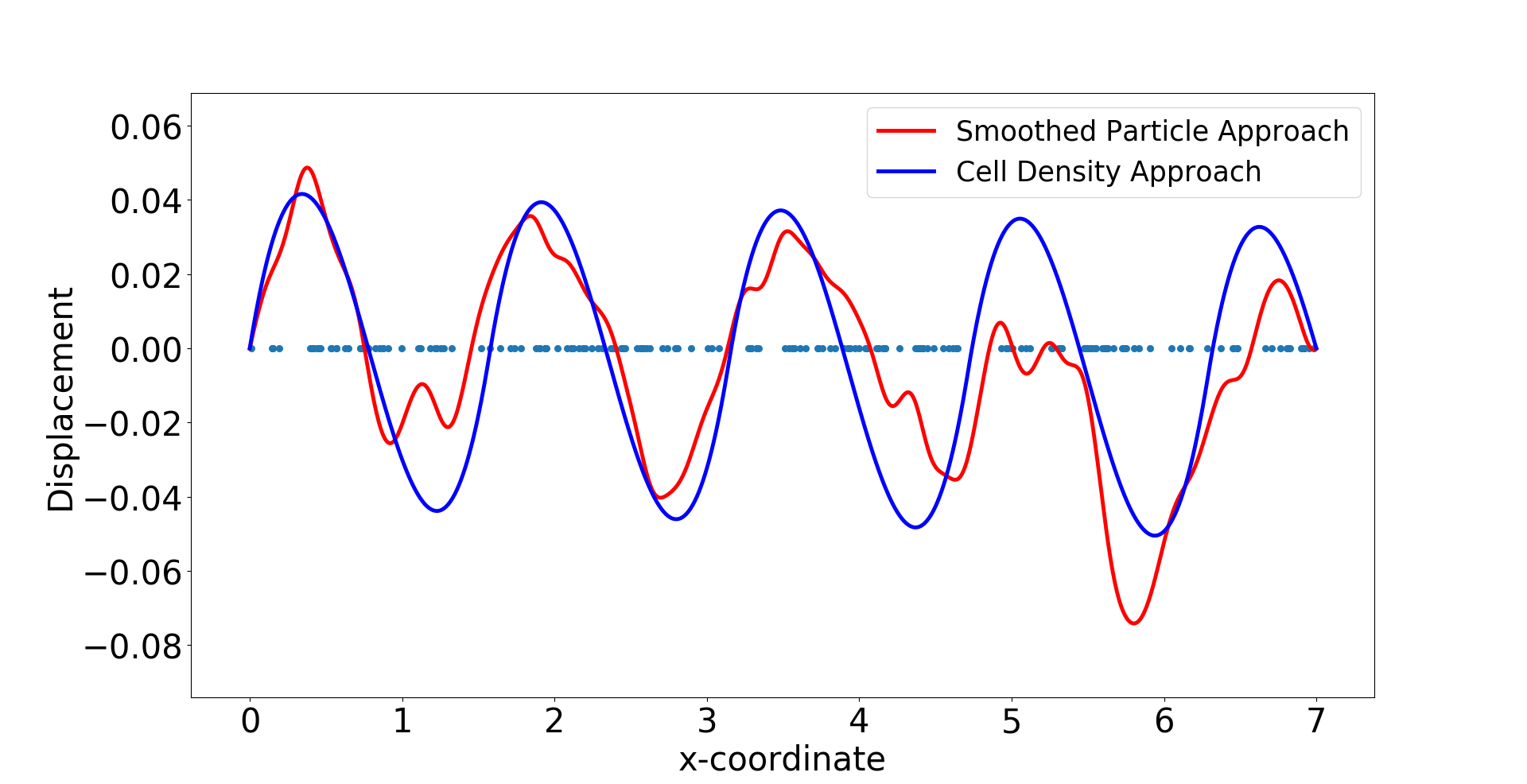}}
	\subfigure[$d(h) = 2h, h = 0.07,d=0.14$]{
		\label{Fig_cell_den_sine_d}
		\includegraphics[width=0.3\textwidth]{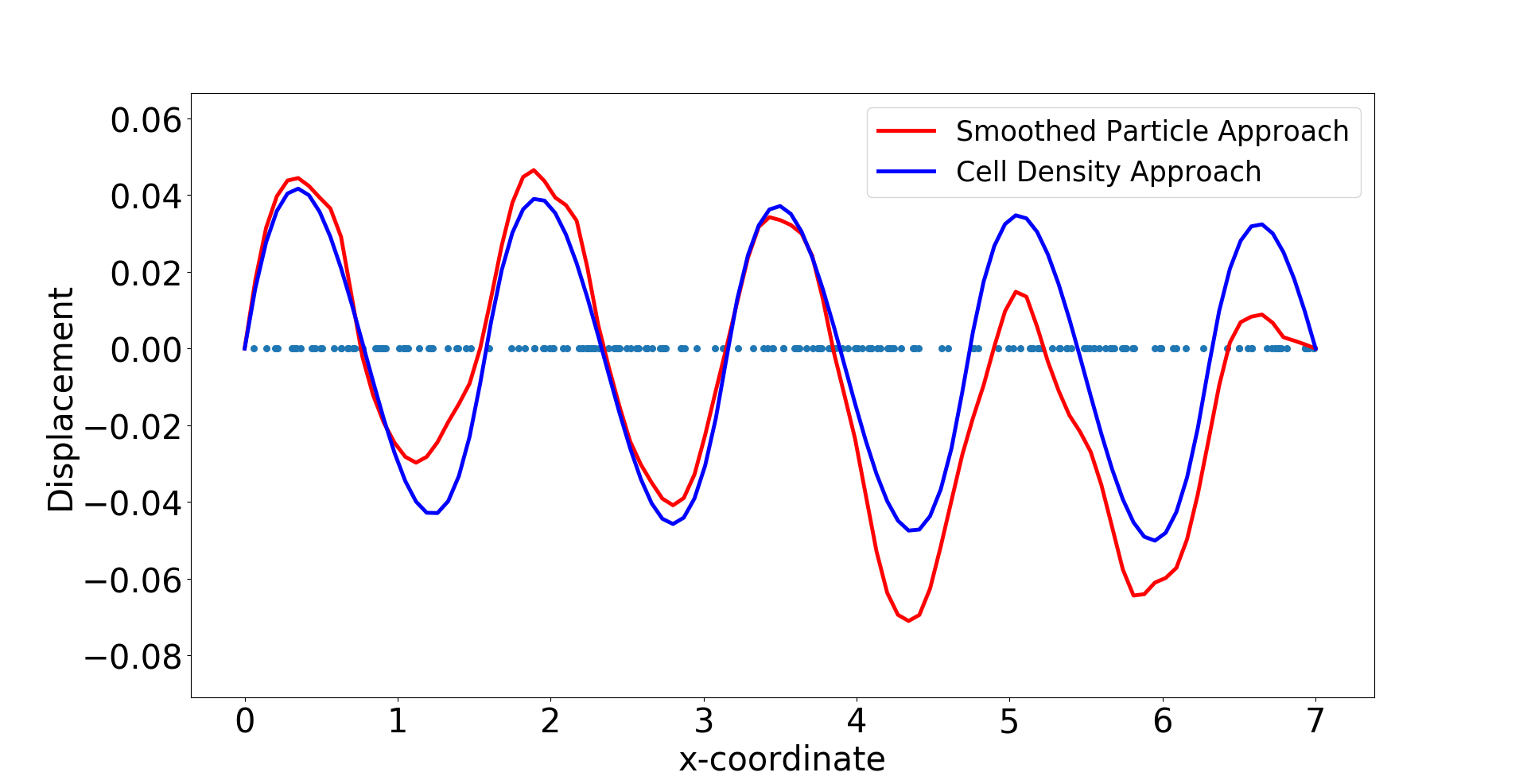}}
	\subfigure[$d(h) = h, h = 0.07,d=0.07$]{
		\label{Fig_cell_den_sine_e}
		\includegraphics[width=0.3\textwidth]{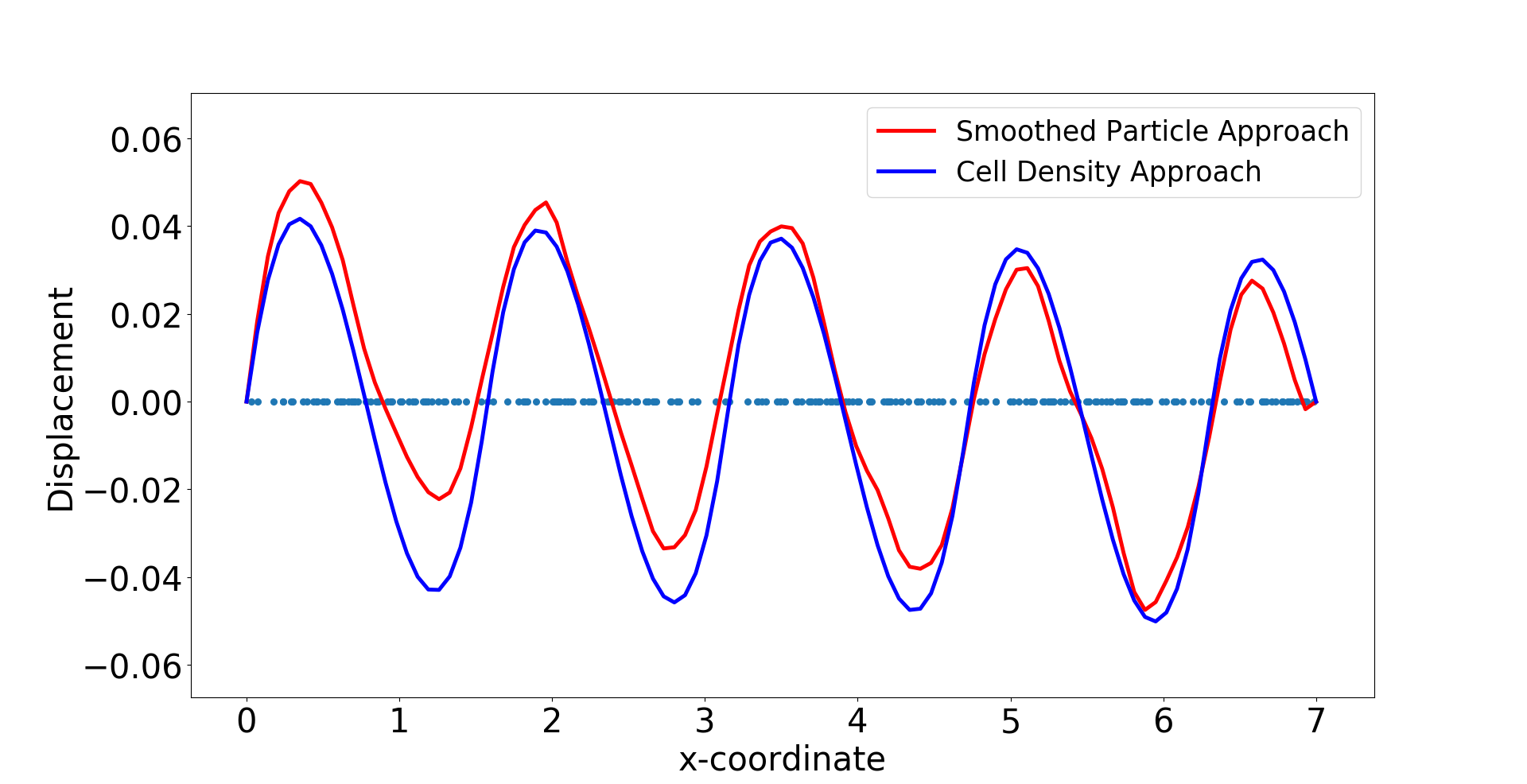}}
	\subfigure[$d(h) = h/4, h = 0.07,d=0.0175$]{
		\label{Fig_cell_den_sine_f}
		\includegraphics[width=0.3\textwidth]{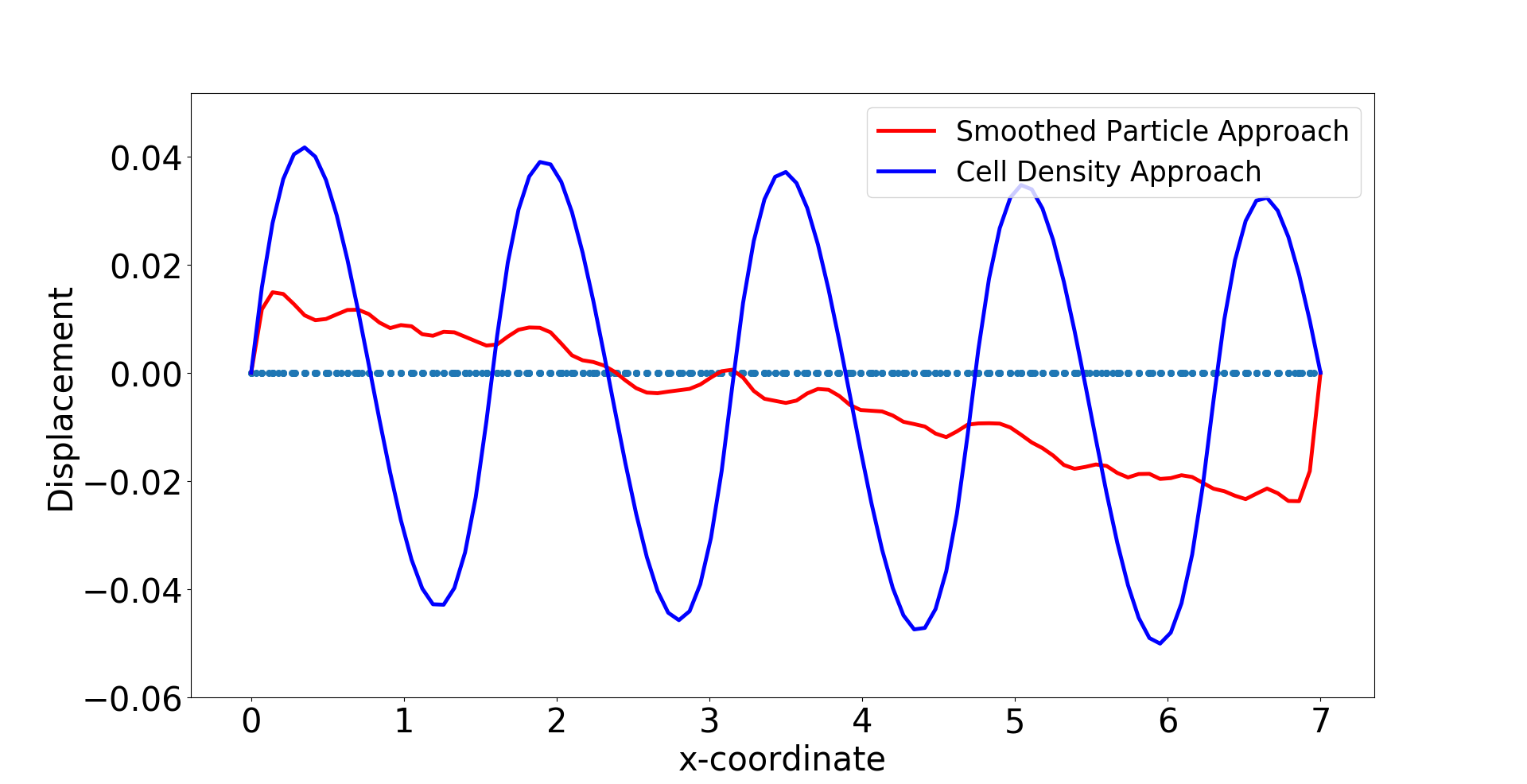}}
	\caption{The cell density function is sine function and using the algorithm in Figure \ref{Diag_cell_density}, different simulations are carried out with various mesh size and the total number of cells. Blue curves represent the solutions to $(BVP^2_{SPH})$, and ref curves are the solutions to $(BVP^2_{den})$ with $n_c(x) = 40|\sin(2x)|$. In Subfigure \subref{Fig_cell_den_sine_a}--\subref{Fig_cell_den_sine_c}, we set $d = 0.35$ and cell positions are fixed. From Subfigure \subref{Fig_cell_den_sine_d} to Subfigure \subref{Fig_cell_den_sine_f}, we use the same finite-element method settings (where $h$ is efficiently small with $h = 0.07$), and we take different values of $d$.}	
	\label{Fig_cell_den_sine}
\end{figure}

We consider cells that are located uniformly in the subdomain $(2,5)$, which implies that the gradient or divergence of  the cell density vanishes inside the subdomain but does not exist at two endpoints of the subdomain. Hence, we utilize the implementation method in Figure \ref{Diag_cell_positions}, as the center positions of the cells are given, then the local cell density can be calculated per unit area. Compared with the results shown in Figure \ref{Fig_1d_linspace_exact}, the results in Figure \ref{Fig_1d_linspace_FEM} and Figure \ref{Fig_1d_linspace_FEM_cells} show the solutions to $(BVP^2_{SPH})$ and $(BVP^2_{den})$ respectively. Note that, in the finite-element method solutions, the magnitude of the forces in both approaches are the same, and the variance of $\delta_{\varepsilon}(x)$ is related to $h$ rather than $\Delta s$. Furthermore, these figures verify that the convergence between SPH approach and cell density approach is determined by the mesh size rather than by the distance between any two adjacent cells. Table \ref{Tbl_1D_Uniform} displays the numerical results of the simulation in Figure \ref{Fig_1d_linspace_FEM}, in the perspective of the solution, the reduction ratio of the subdomain and the computational cost. Similarly to the figures, there is no significant difference between the norms and the deformed length of the subdomain. However, the simulation time in the cell density approach is much shorter than in the SPH approach with a factor of $35$.
\begin{figure}[htpb]
	\centering
	\subfigure[$h = 0.7$]{
		\includegraphics[width=0.45\textwidth]{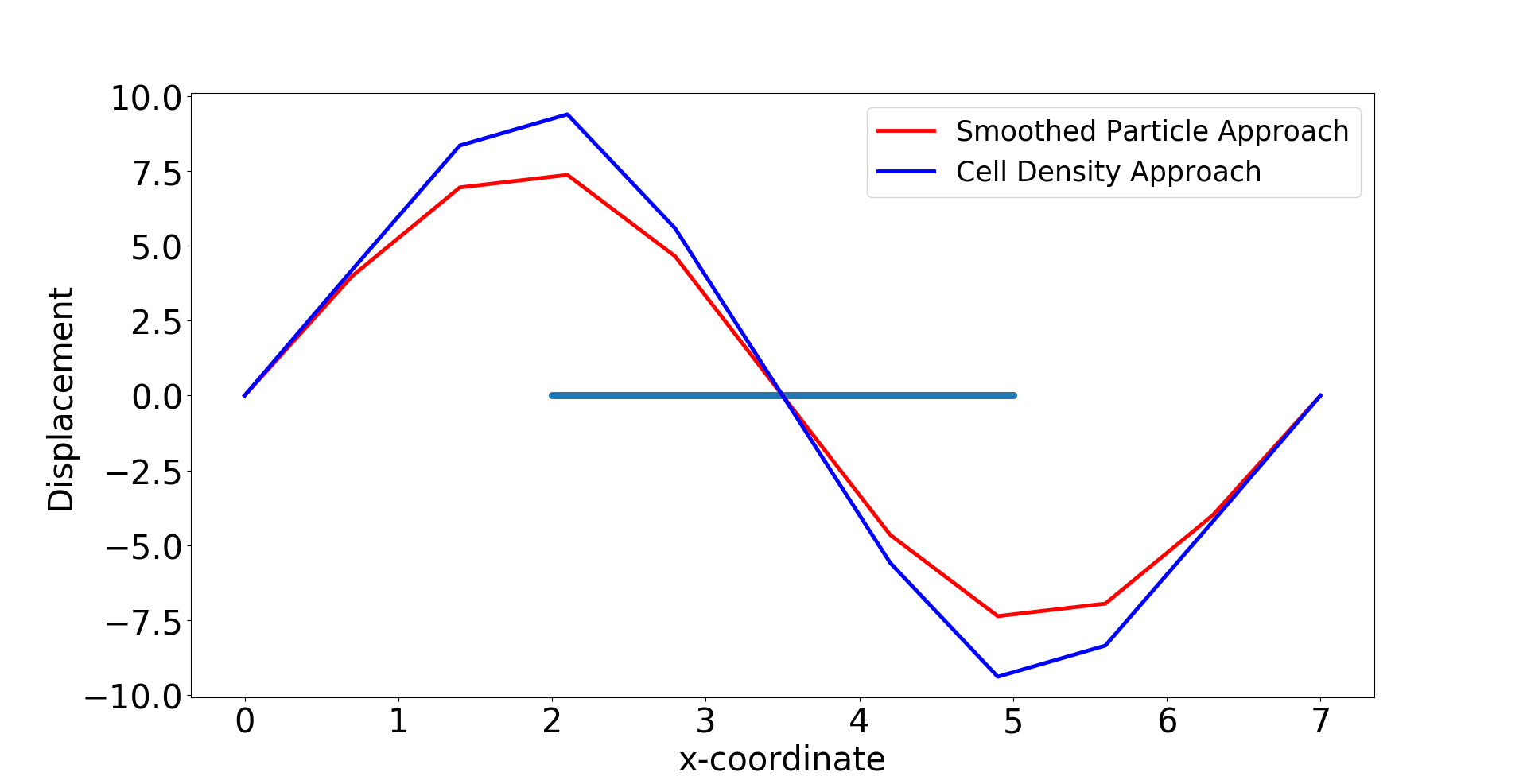}}
	\subfigure[$h = 0.07$]{
		\includegraphics[width=0.45\textwidth]{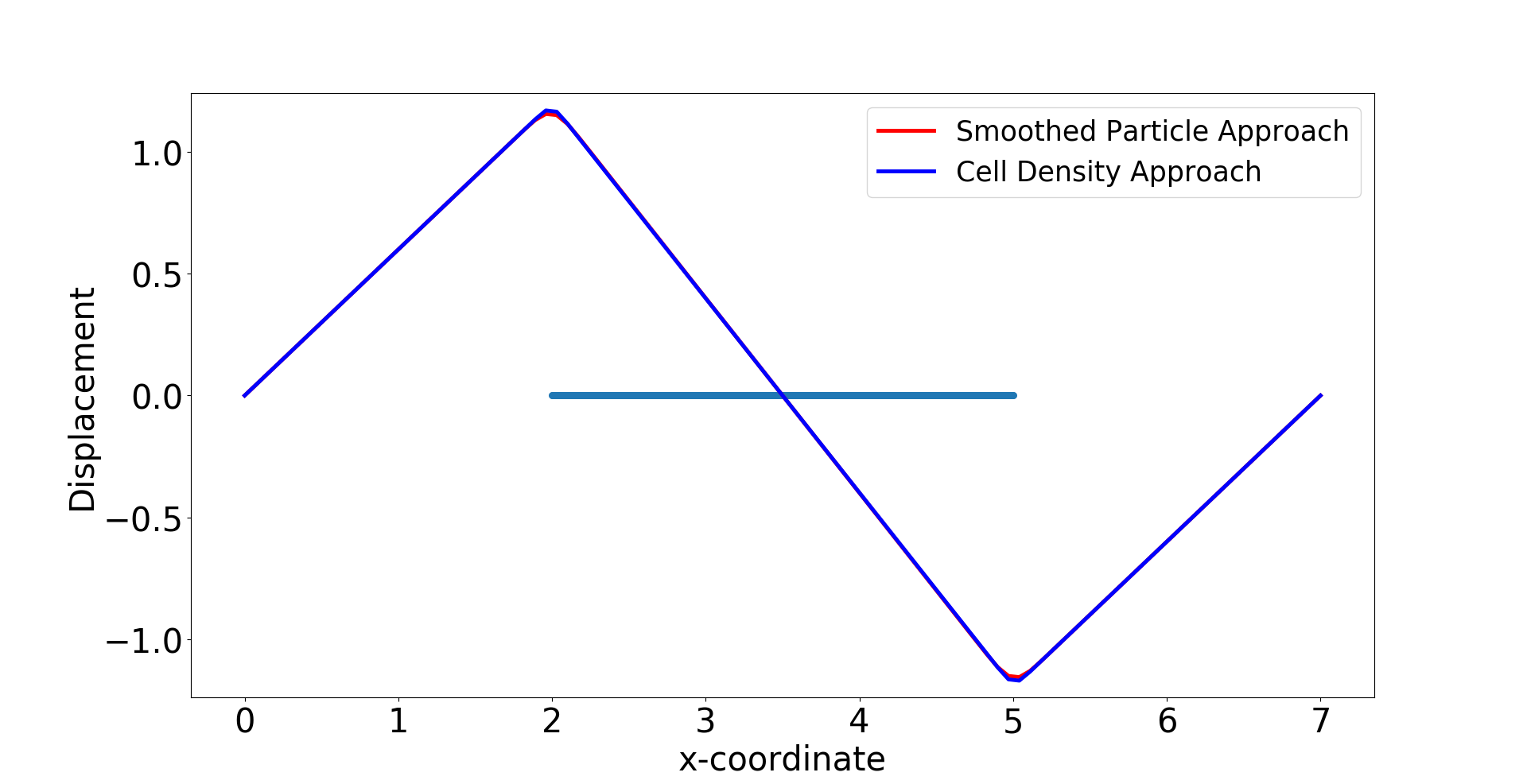}}
	\subfigure[$h = 0.014$]{
		\includegraphics[width=0.45\textwidth]{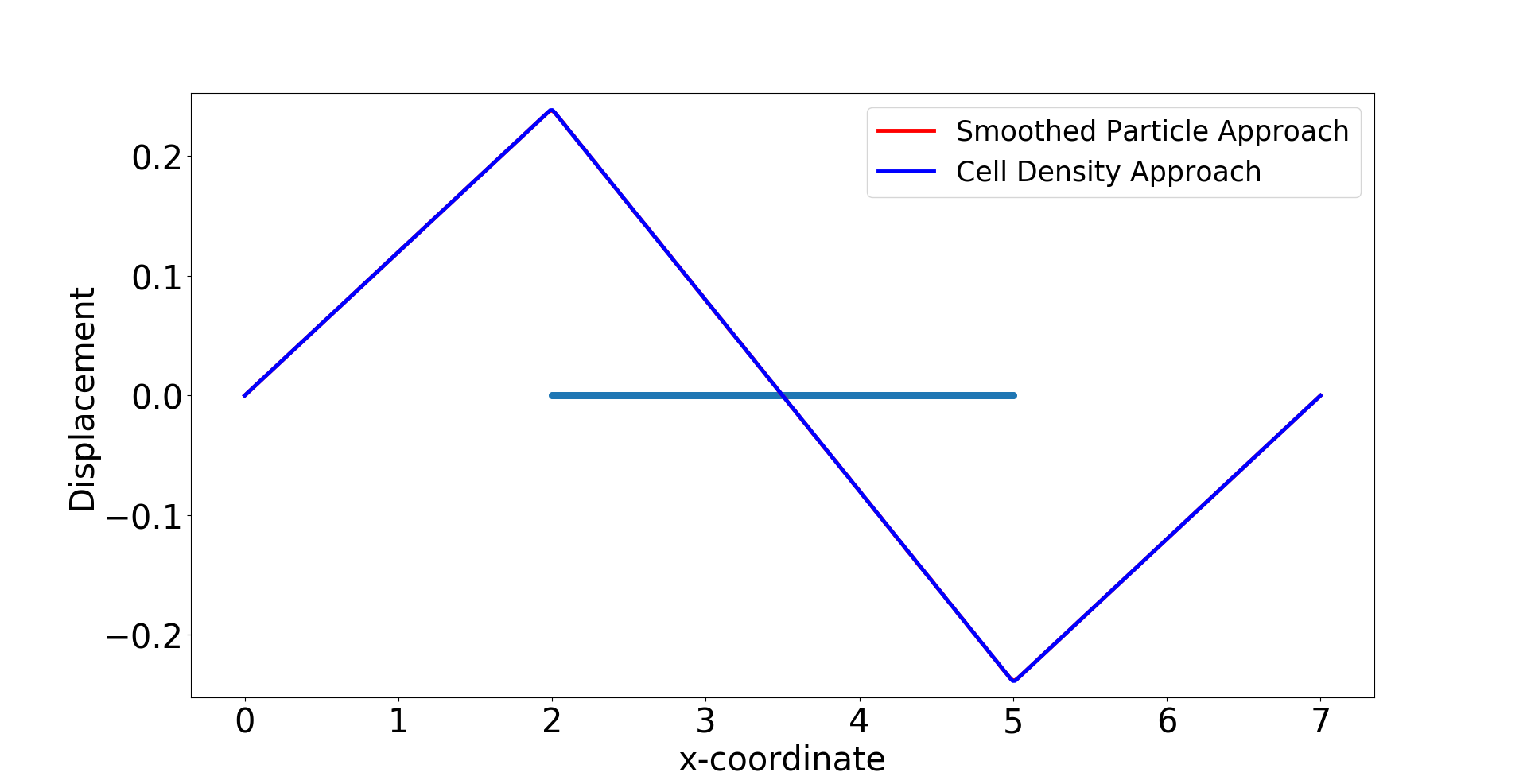}}
	\subfigure[$h = 0.007$]{
		\includegraphics[width=0.45\textwidth]{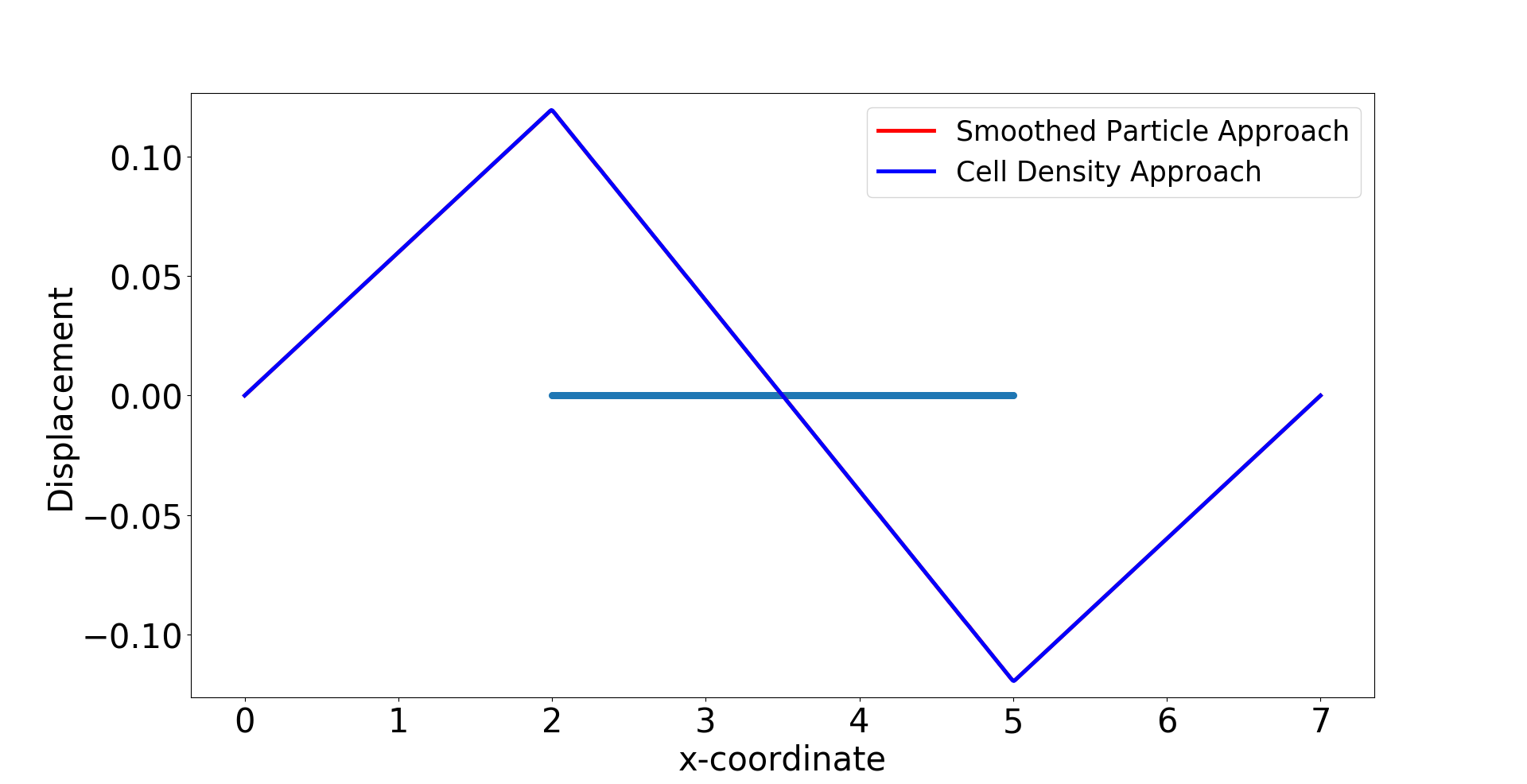}}
	\caption{The finite-element method solutions to $(BVP^2_{SPH})$ and $(BVP^2_{den})$ are shown where cells are uniformly located. With the fixed positions of cells, the solutions are convergent as $h \rightarrow 0^+$. Blue points are the centre positions of biological cells. Red curves represent the solutions to $(BVP^2_{SPH})$ and blue curves represent the solutions to $(BVP^2_{den})$.}
	\label{Fig_1d_linspace_FEM}
\end{figure}

\begin{figure}[htpb]
		\centering
	\subfigure[$N_s = 10$]{
		\includegraphics[width=0.45\textwidth]{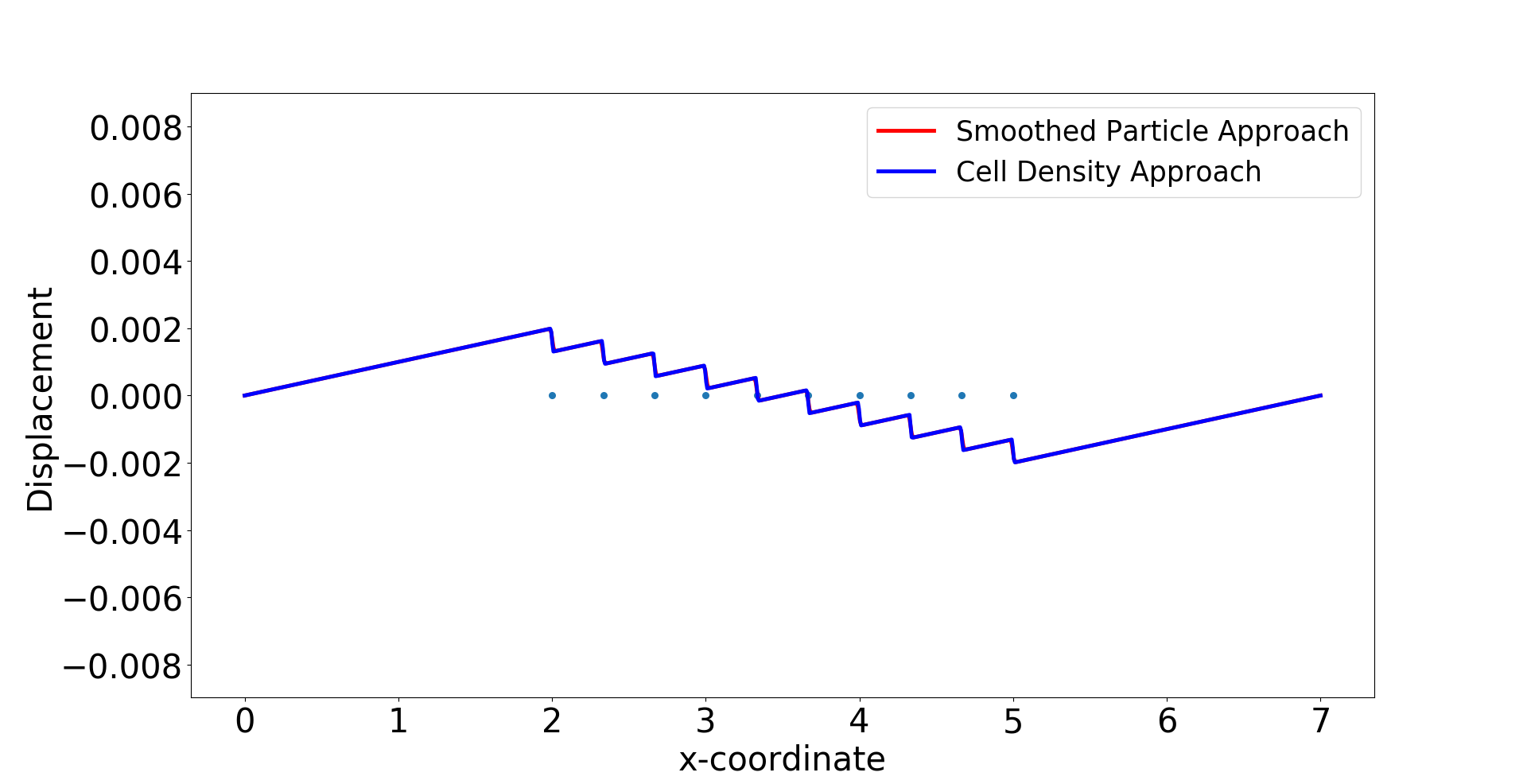}}
	\subfigure[$N_s = 50$]{
		\includegraphics[width=0.45\textwidth]{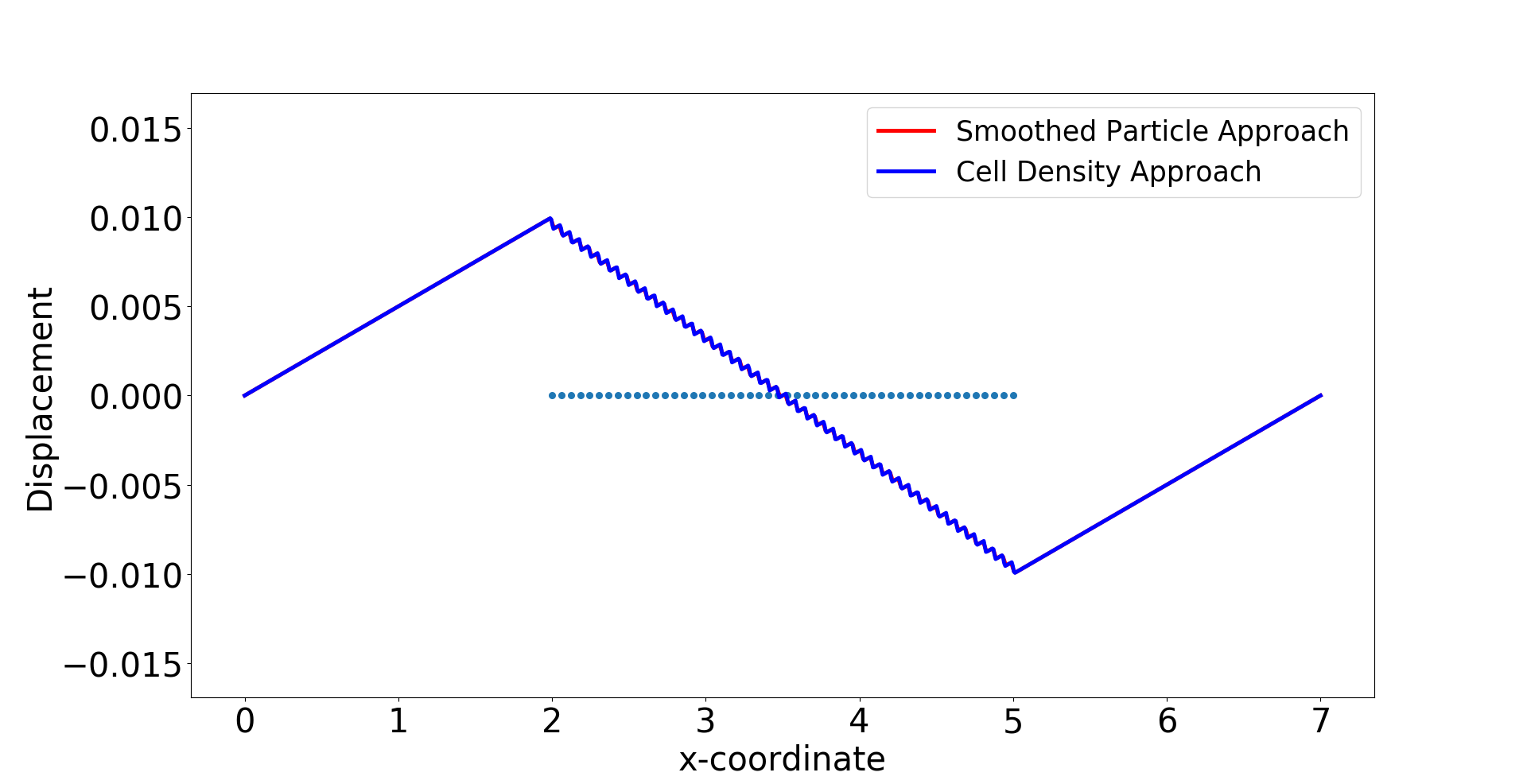}}
	\subfigure[$N_s = 100$]{
		\includegraphics[width=0.45\textwidth]{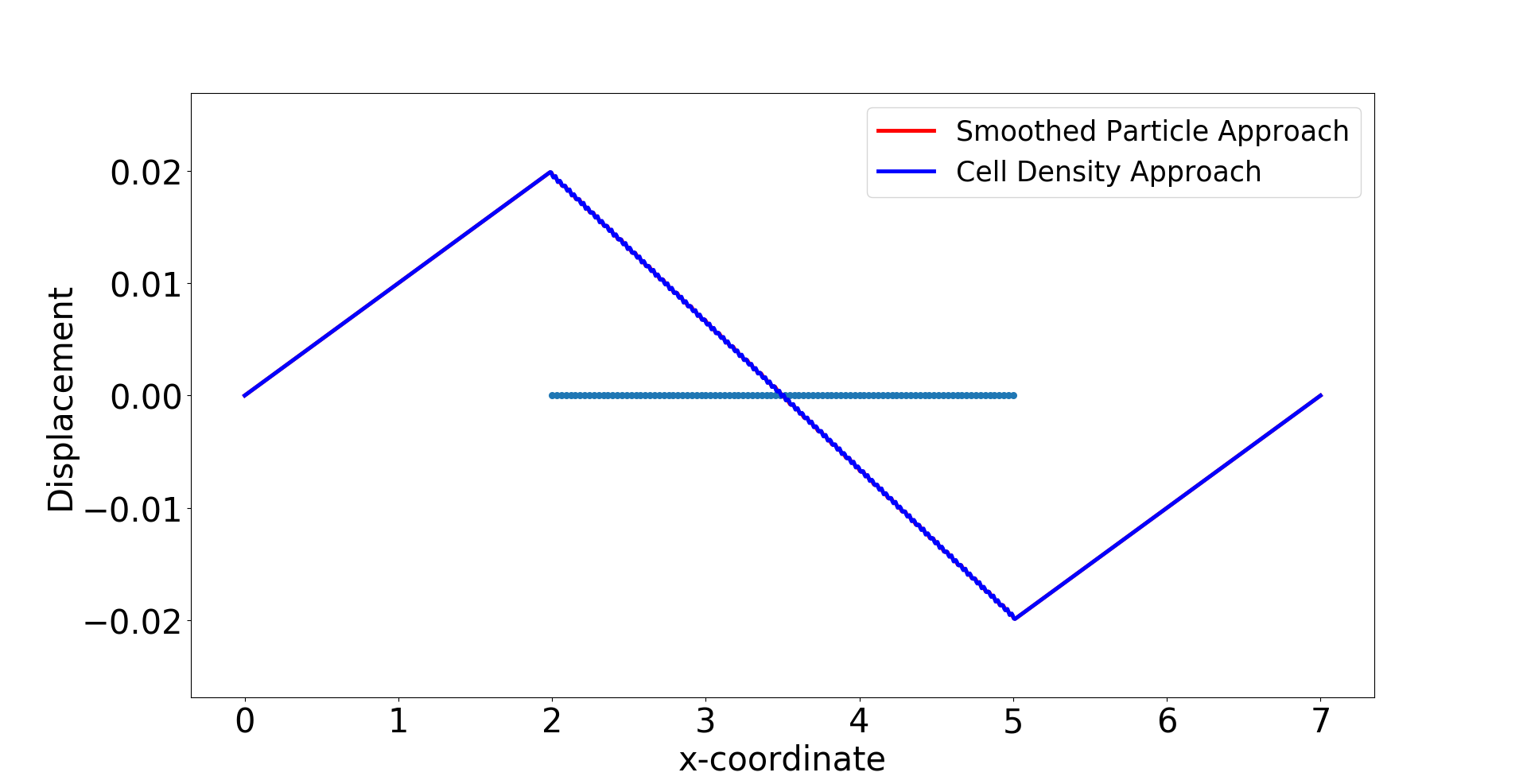}}
	\subfigure[$N_s = 500$]{
		\includegraphics[width=0.45\textwidth]{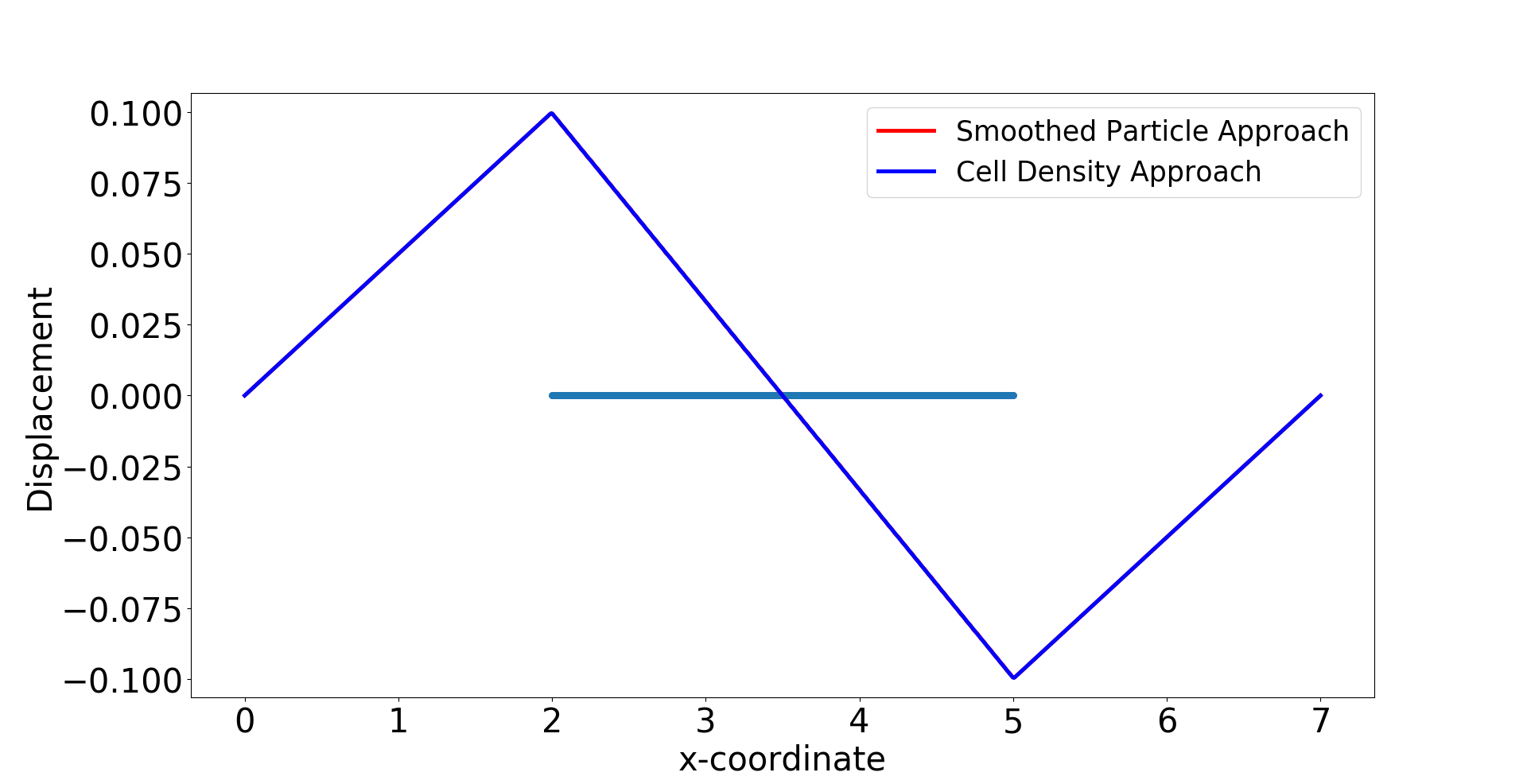}}
	\caption{The finite-element method solutions to $(BVP^2_{SPH})$ and $(BVP^2_{den})$ are shown with uniform distribution. Compared to the analytical result, the consistency between two approaches are unrelated to the number of cells, and  the solutions are convergent as $h \rightarrow 0^+$. Here, we use $h = 0.007$. Blue points are the centre positions of biological cells. Red curves represent the solutions to $(BVP^2_{SPH})$ and blue curves represent the solutions to $(BVP^2_{den})$.}
	\label{Fig_1d_linspace_FEM_cells}
\end{figure}

\begin{table}[htpb]
	\centering
	\caption{Numerical results of two approaches in one dimension with biological cells located uniformly. Here, we define mesh size $h = 0.07$ and $N_s = 50$, which means $\Delta s = 0.06$. The results are solved by finite-element method with algorithm in Figure \ref{Diag_cell_positions}.}
	\begin{tabular}{m{6cm}<{\centering}m{5cm}<{\centering}m{5cm}<{\centering}}
		\toprule
		& {\bf SPH Approach} & {\bf Cell Density Approach}\\
		\midrule
		$\|u\|_{L^2((0,L))}$ & $0.1858655201$ & $0.1858660118$\\
		$\|u\|_{H^1((0,L))}$ & $0.2780804415$ & $0.2914497482$\\
		Convergence rate of $L^2-norm$ & $0.9940317098$ & $0.9985295706$\\
		Convergence rate of $H^1-norm$ & $1.002001685$ & $1.004380036$\\
		Reduction ratio of the subdomain $(a,b)$ (\%) & $7.96908$ & $7.98821$ \\
		Time cost $(s)$ & $0.10391$ & $0.0030458$\\ 
		\bottomrule
	\end{tabular}
	\label{Tbl_1D_Uniform}
\end{table}

\subsection{Two-Dimensional Results}
\noindent
In the multi-dimensional case, we are not able to write the analytical solution to the boundary value problems. The results are all solved by the use of the finite-element method applied to $(BVP^3_{SPH})$ and $(BVP^3_{den})$. Note that the force magnitude of both boundary value problems is the same. Following the same implementation methods as in one dimension, simulations are carried out with two formulas of cell density: (1) cells are located inside the subdomain $\Omega_w$ randomly by the uniform distribution; (2) the cell density function is in the form of the standard Gaussian distribution over the computational domain with $n_c(\boldsymbol{x}) = 50\times\frac{1}{2\pi}\exp\left\{-\frac{\|\boldsymbol{x}|^2}{2}\right\}$. Implementation methods in Figure \ref{Diag_cell_positions} and Figure \ref{Diag_cell_density} are applied respectively in Simulation (1) and (2).

For Simulation (1), no analytical expression for (the derivative of) the density function is available. Therefore, the implementation starts with generating the cell positions, according to the principles outlined in Figure \ref{Diag_cell_positions}. In Figure \ref{Fig_2D_random}, the displacement results are displayed. From the figures, hardly any significant differences between the solutions can be observed, which indicates that these two approaches are numerically consistent. Table \ref{Tbl_2D_random} displays more details about the two approaches regarding the numerical analysis: most data are more or less the same. Thanks to the continuity of the SPH approach, the convergence rate between two approaches is similar. However, as it has been mentioned earlier, the agent-based model is computationally more expensive than the continuum-based model; here, the difference is a factor of $240$.

\begin{figure}[htpb]
\centering 
\includegraphics[width=0.75\textwidth]{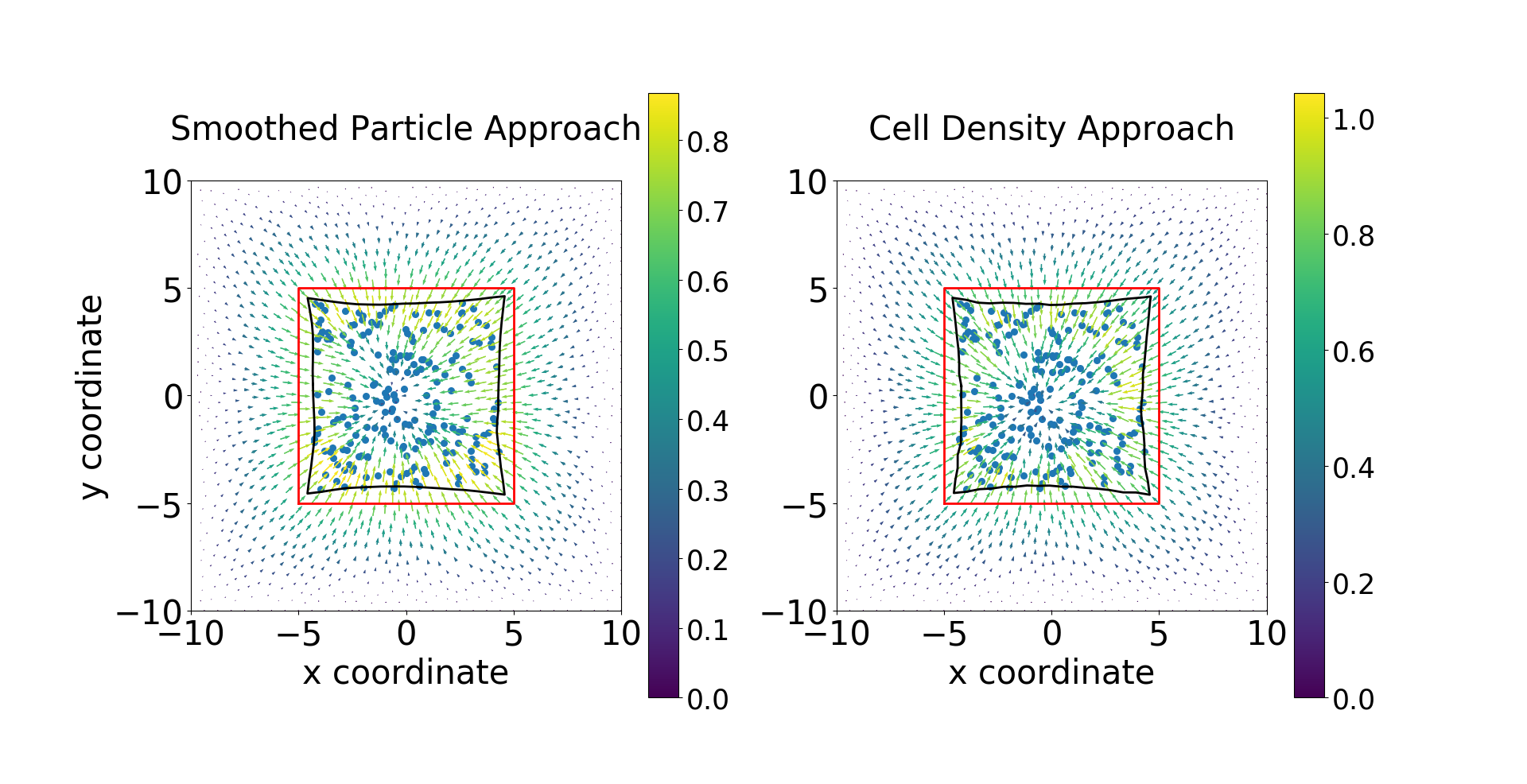}
\caption{The displacement results are shown, which are solved from $(BVP^3_{SPH})$ and $(BVP^3_{den})$ when cells are randomly located in the subdomain $(-5,5)\times(-5,5)$. In other words, it is impossible to write the analytical expression of $n_c(\boldsymbol{x})$, subsequently, the algorithm in Figure \ref{Diag_cell_positions} is selected. There are $196$ biological cells in the computational domain. Blue points are the center positions of biological cells, red curves are the original shapes of the subdomain, and the black curves represent the deformed boundary of the subdomain.}
\label{Fig_2D_random}
\end{figure}

\begin{table}[htpb]
	\centering
\caption{Numerical results of two approaches in two dimensions with random distribution for the positions of biological cells. Due to the nonexistence of divergence or gradient of cell density function, implementation method in Figure \ref{Diag_cell_positions} is used.}
\begin{tabular}{m{6cm}<{\centering}m{5cm}<{\centering}m{5cm}<{\centering}}
	\toprule
	& {\bf SPH Approach} & {\bf Cell Density Approach}\\
	\midrule
	$\|u\|_{L^2((0,L))}$ & $10.105858093727422$ & $12.314518769308366$\\
	$\|u\|_{H^1((0,L))}$ & $10.890035548437087$ & $16.67100195050636$\\
	Convergence rate of $L^2-norm$ & $2.012144067$ & $2.061050246$\\
	Convergence rate of $H^1-norm$ & $2.011995299$ & $2.03760311$\\
	Reduction ratio of the subdomain $\Omega_w$ (\%) & $24.24192$ & $23.33667$ \\
	Time cost $(s)$ & $4.20677$ & $1.77219\times10^{-2}$\\ 
	\bottomrule
\end{tabular}
\label{Tbl_2D_random}
\end{table}

According to the setting of the simulation, we define the cell density function by $$n_c(\boldsymbol{x}) = 50\times\frac{1}{2\pi}\exp\left\{-\frac{\|\boldsymbol{x}|^2}{2}\right\}, \mbox{in $\Omega$},$$ which is a Gaussian distribution multiplied by a positive constant. Similarly, in two dimensions, the cell density is defined by the number of biological cells per unit area. In other words, the cell count is computed by the local cell density multiplied by the area of selected region. Here, we assume that the selected region is a $1\times1$ square, then we generate the center positions of biological cells in every unit square based on the local number of cells. Figure \ref{Fig_2D_Gaussian} shows the numerical results regarding two approaches. There is no significant difference if the same mesh structure resolution is used. As the mesh is refined, the solution to the SPH approach is smoother, since the "ring" in the center becomes more dominant. In Table \ref{Tbl_2D_Gaussian}, it can be concluded that there are no significant differences between two approaches except for the computational efficiency and the convergence rate of $H^1$-norm. If the mesh is not fine enough, then the solution to the SPH approach is less smooth, hence, the determination of the gradient of the solution is less accurate.

\begin{figure}[htpb]
	\centering
	\subfigure[$h = 0.6435$]{
	\includegraphics[width=0.75\textwidth]{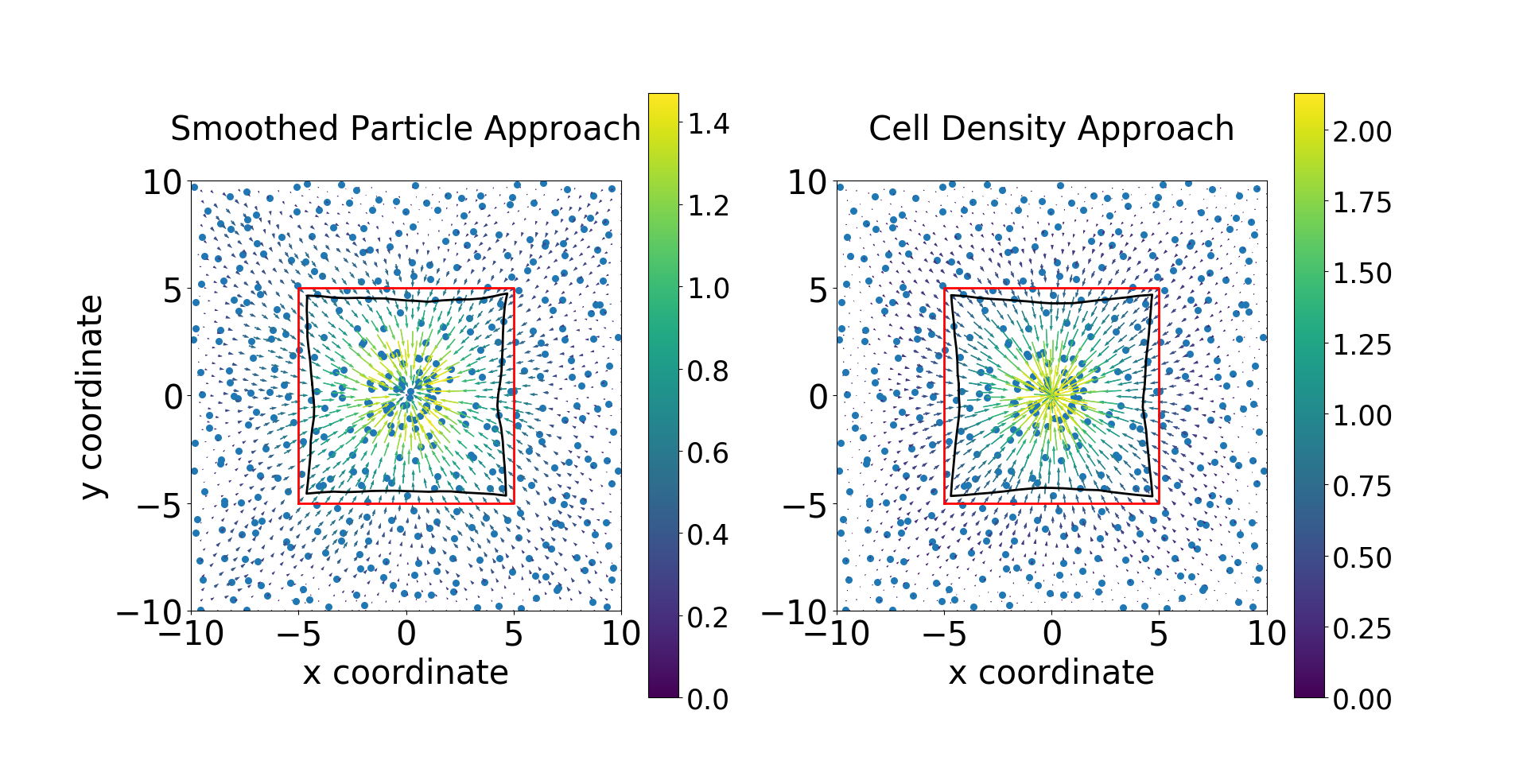}}
    \subfigure[$h = 0.3218$]{
	\includegraphics[width=0.75\textwidth]{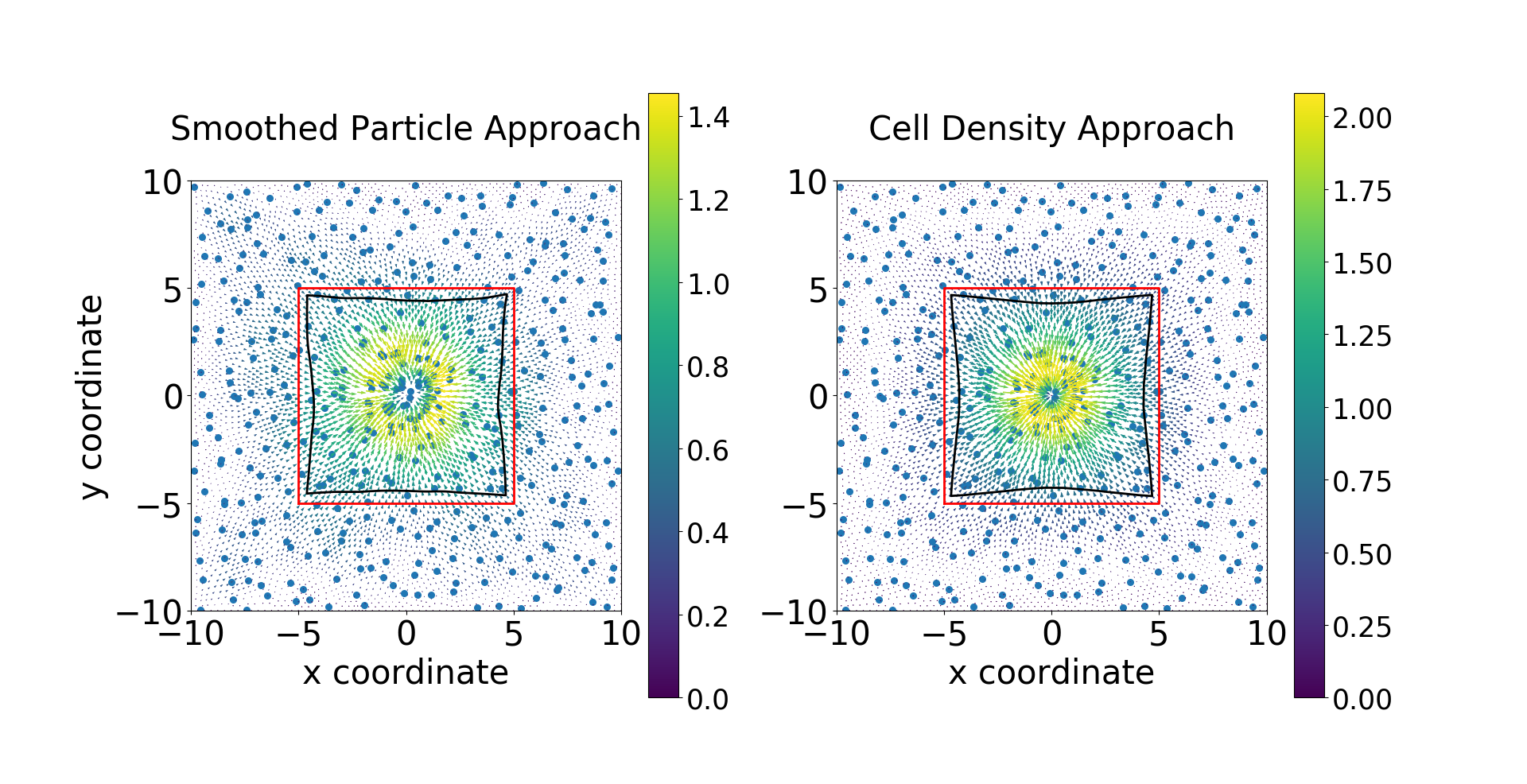}}
	\subfigure[$h = 0.1609$]{
	\includegraphics[width=0.75\textwidth]{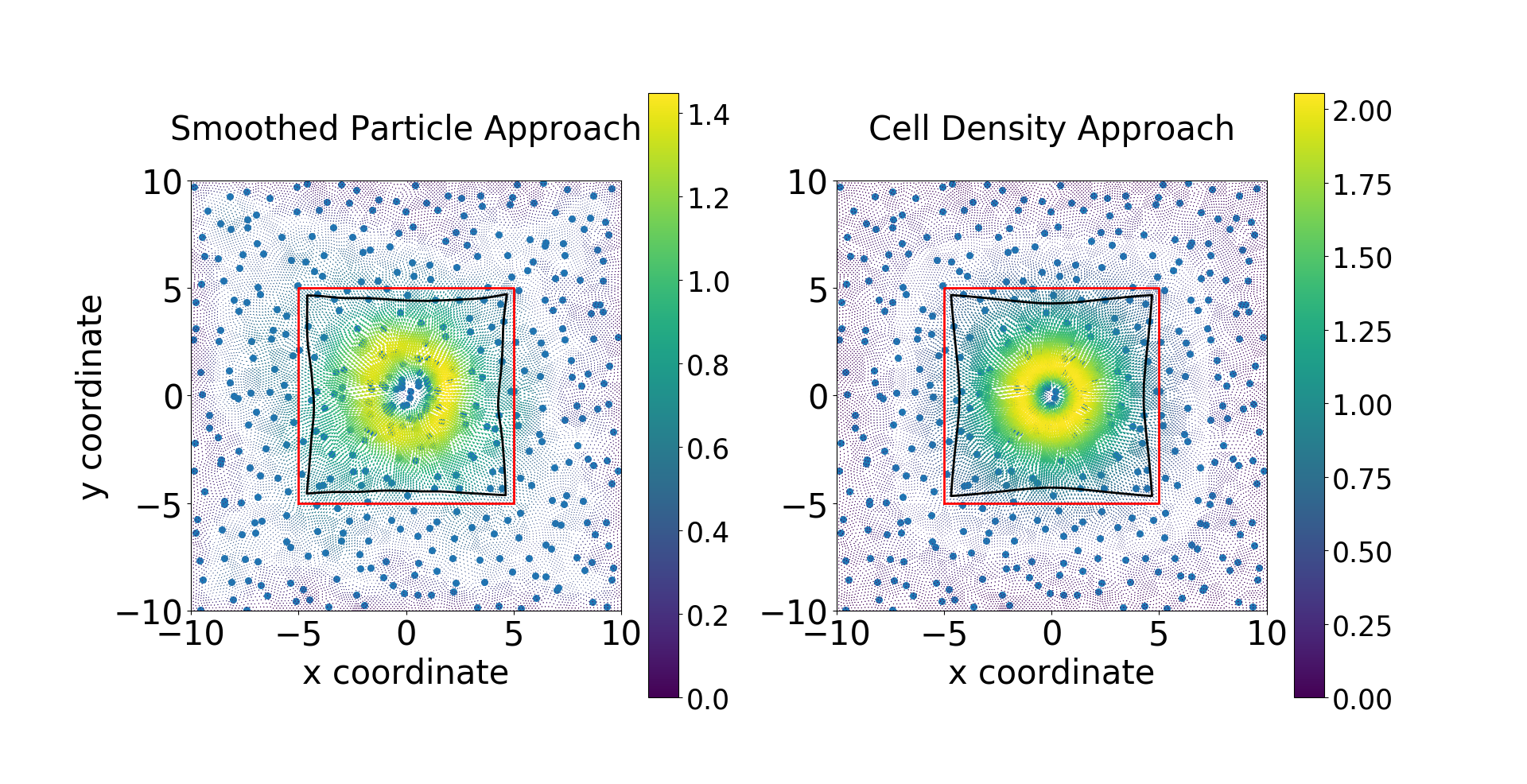}}
	\caption{The displacement results are shown, which are solved from $(BVP^3_{SPH})$ and $(BVP^3_{den})$ when cells are located according to the cell density function $n_c(\boldsymbol{x}) = 50\times\frac{1}{2\pi}\exp\{-\frac{\|\boldsymbol{x}|^2}{2}\}$. Hence, implementing algorithm in Figure \ref{Diag_cell_positions} is used. There are $440$ biological cells in the computational domain. Blue points are the center positions of biological cells, red curves are the original shapes of the subdomain, and the black curves represent the deformed boundary of the subdomain.}
	\label{Fig_2D_Gaussian}
\end{figure}

\begin{table}[htpb]
	\centering
	\caption{Numerical results of two approaches in two dimensions with Gaussian distribution for the positions of biological cells. Figure \ref{Diag_cell_density} is implemented and there are $440$ biological cells in the computational domain.}
	\begin{tabular}{m{6cm}<{\centering}m{5cm}<{\centering}m{5cm}<{\centering}}
		\toprule
		& {\bf SPH Approach} & {\bf Cell Density Approach}\\
		\midrule
		$\|u\|_{L^2((0,L))}$ & $11.14304909846569$ & $13.188441094735877$\\
		$\|u\|_{H^1((0,L))}$ & $12.77795210802095$ & $15.533928099123479$\\
		Convergence rate of $L^2-norm$ & $1.724918322$ & $1.856132219$\\
		Convergence rate of $H^1-norm$ & $3.059831229$ & $2.455660849$\\
		Reduction ratio of the subdomain $\Omega_w$ (\%) & $19.65854$ & $20.55949$ \\
		Time cost $(s)$ & $0.62347$ & $0.017315$\\ 
		\bottomrule
	\end{tabular}
	\label{Tbl_2D_Gaussian}
\end{table}

\section{Conclusions}\label{Conclusions}
\noindent
In this manuscript, we discussed the different models to simulate the pulling forces exerted by the (myo)fibroblasts depending on different scales of the wound region. We started from one dimension and later extended the models to two dimensions. In one dimension, we can write explicitly the analytical solution to the boundary value problem with specific distribution of the locations of biological cells. The consistency of the solutions indicates the possibility to prove the analytical connection between the two approaches. In both one and two dimensions,  the numerical solutions delivered by finite-element methods with Lagrange linear basis functions implied that these two models are consistent under certain mesh conditions (when the mesh size is sufficiently small) and regardless the locations of the biological cells and the implementation methods. In summary, regarding the displacement of the ECM from the mechanical model, the agent-based model and the cell density model are consistent from a computational point of view. This could be used to transfer one type of model to the other one regarding the force balance in the wound healing model, as the connection between these two models has been suggested. We want to use the developed insights for the analysis of upscaling between agent-based and continuum-based model formulations.

\newpage
\bibliographystyle{abbrvnat}
\bibliography{SPH_cell_density}

\end{document}